\documentclass[11pt,reqno]{amsart}

\usepackage{amsmath, amsthm, amssymb}
\usepackage{amsfonts}
\usepackage[ansinew]{inputenc}
\usepackage[dvips]{epsfig}
\usepackage{graphicx}
\usepackage[english]{babel}

\usepackage{tcolorbox}
\usepackage{xcolor}

\usepackage{cite}
\usepackage{graphicx}
\usepackage{amscd}
\usepackage{xcolor}
\usepackage{bm}
\usepackage{enumerate}

\usepackage{verbatim}
\usepackage{hyperref}
\usepackage{amstext}
\usepackage{latexsym}
\theoremstyle{plain}
\newtheorem{theorem}{Theorem}[section]
\newtheorem{definition}[theorem]{Definition}
\newtheorem{corollary}[theorem]{Corollary}
\newtheorem{proposition}[theorem]{Proposition}
\newtheorem{lemma}[theorem]{Lemma}
\newtheorem{remark}[theorem]{Remark}

\numberwithin{theorem}{section}
\numberwithin{equation}{section}

\newcommand{\average}{{\mathchoice {\kern1ex\vcenter{\hrule height.4pt
width 6pt depth0pt} \kern-9.7pt} {\kern1ex\vcenter{\hrule
height.4pt width 4.3pt depth0pt} \kern-7pt} {} {} }}

\def\R{\mathbb{R}}

\def\div{\text{div}}

\renewcommand{\a }{\alpha }
\renewcommand{\b }{\beta }
\renewcommand{\d}{\delta }

\newcommand{\D }{\Delta }

\newcommand{\e }{\varepsilon }
\newcommand{\g }{\gamma}
\renewcommand{\i }{\iota}
\newcommand{\G }{\Gamma}
\renewcommand{\l }{\lambda }
\renewcommand{\L }{\Lambda }
\newcommand{\n }{\nabla }

\newcommand{\s }{\sigma }

\newcommand{\z }{\zeta}
\renewcommand{\th }{\theta }

\renewcommand{\O }{\Omega }

\newcommand{\ov}{\overline}

\newcommand{\be}{\begin{equation}}
\newcommand{\ee}{\end{equation}}

\newcommand{\de}{\partial}

\newcommand{\ti}{\widetilde}

\renewcommand{\k}{\kappa}

 \usepackage{mathrsfs}

\newcommand{\calC }{\mathcal{C}}

\newcommand{\calD }{\mathcal{D}}

\newcommand{\N}{\mathbb{N}}

\renewcommand{\H}{{\mathcal H}}

\newcommand{\cC}{{\mathcal C}}

\newcommand{\cF}{{\mathcal F}}

\newcommand{\cH}{{\mathcal H}}

\newcommand{\cK}{{\mathcal K}}
\newcommand{\cL}{{\mathcal L}}
\newcommand{\cM}{{\mathcal M}}

\newcommand{\cO}{{\mathcal O}}
\newcommand{\cP}{{\mathcal P}}

\newcommand{\dist}{{\rm dist}}

\newcommand{\eps}{\varepsilon}

\renewcommand{\epsilon}{\varepsilon}

\newcommand{\Ds}{ (-\D)^s}

\DeclareMathOperator{\spann}{span}

\begin{document}
\title[Nondegeneracy properties and uniqueness of positive solutions]{Nondegeneracy properties and uniqueness of positive solutions to a class of fractional semilinear equations}

 \author[]
{ Mouhamed Moustapha Fall and Tobias Weth}


\address{${}^1$African Institute for Mathematical Sciences in Senegal (AIMS Senegal), 
KM 2, Route de Joal, B.P. 14 18. Mbour, S\'en\'egal.}
\address{${}^2$Goethe-Universit\"{a}t Frankfurt, Institut f\"{u}r Mathematik. Robert-Mayer-Str. 10, D-60629 Frankfurt, Germany.}

\email{weth@math.uni-frankfurt.de}
\email{mouhamed.m.fall@aims-senegal.org}

 \begin{abstract}
   \noindent 
We prove that positive solutions  $u\in H^s(\mathbb{\R}^N)$ to the equation  $(-\Delta )^s u+ u=u^p$  in $\mathbb{\R}^N$  are nonradially nondegenerate,  for all $s\in (0,1)$,  $N\geq 1$ and  $p>1$ strictly smaller than the critical Sobolev exponent. 
By this we mean that the linearized equation $(-\Delta )^s w+ w-pu^{p-1}w = 0$ does not admit nonradial solutions beside the directional derivatives of $u$.  \\
Letting $B$ be the unit  centered ball and $\l_1(B)$ the first Dirichlet eigenvalue of the fractional Laplacian $(-\Delta )^s$,   we also prove  that positive solutions to  $(-\Delta )^s u+\lambda u=u^p$  in   ${B}$ with   $u=0$ on $\mathbb{\R}^N\setminus B$,  are nonradially nondegenerate for any   $\lambda> -\l_1(B)$ in the sense that the linearized equation does not admit nonradial solutions. From these results, we then deduce uniqueness and full nondegeneracy of positive solutions in some special cases.
In particular, in the case  $N=1$, we prove that the equation $(-\Delta )^s u+ u=u^2$ in $\R$ or in $B$, with zero exterior data,  admits a unique even solution which is fully nondegenerate in the optimal range $s \in (\frac{1}{6},1)$, thus extending the classical  uniqueness result of Amick and Toland on the Benjamin-Ono equation. Moreover, in the case $N=1$, $\lambda=0$, we also prove the uniqueness and full nondegeneracy of positive solutions for the Dirichlet problem in $B$ with arbitrary subcritical exponent $p$. 
Finally, we determine the unique positive ground state solution of $(-\Delta )^{\frac{1}{2}} u+ u=u^{p}$ in $\mathbb{\R}^N$, $N \ge 1$ with $p=1+\frac{2}{N+1}$and compute the sharp constant in the associated Gagliardo-Nirenberg inequality
  $$
  \|u\|_{L^{p+1}(\R^N)} \le C \|(-\Delta )^{\frac{1}{4}} u\|_{L^2(\R^N)}^{\frac{N}{N+2}} \|u\|_{L^2(\R^N)}^{\frac{2}{N+2}}.
  $$
 \end{abstract}

\maketitle
\setcounter{equation}{0}

\section{Introduction}\label{s:Intro}
Let $s\in (0,1)$ and $N\geq 1$.   The present paper is devoted to nondegeneracy properties of solutions to the problem
  \begin{align}\label{eq:1.1RN}
\Ds u+ \lambda u  = u^{p}\quad\text{ in $\R^N$},\qquad u>0 \quad\textrm{ in $\R^N$}, \qquad  u\in H^s(\R^N)
\end{align} 
and to the Dirichlet problem
\begin{align}\label{eq:1.1}
\Ds u+\l u  = u^{p}\quad\text{ in ${B}$},\qquad u>0 \quad\textrm{in ${B}$}, \qquad  u\in \cH^s(B)
\end{align} 
where $B:=\{x\in \R^N\,:\, |x|<1\}$ is the unit ball. Here $\Ds$ denotes the fractional Laplacian of order $s$, which, under appropriate smoothness and integrability assumptions on the function $u$, is pointwisely given by
$$
\Ds u (x) = c_{N,s} \lim_{\epsilon\to0^+} \int_{\R^N\setminus
   B_\epsilon(x) }\frac{u(x)-u(y)}{|x-y|^{N+2 s}}  dy 
 $$
 with $c_{N,s}=2^{2s}\pi^{-\frac{N}{2}}s\frac{\Gamma(\frac{N+2s}2)}{\Gamma(1-s)}$. Moreover, $H^s(\R^N)$ is the usual fractional Sobolev space endowed with the scalar product 
  $$
(v_1,v_2) \mapsto \langle v_1,v_2 \rangle_{s} = [v_1,v_2]_s +\int_{\R^N}v_1 v_2dx  
$$
and induced norm $v\mapsto \|v\|_s=\sqrt{[v]^2_s+\|v\|^2_{L^2(\R^N)}}$, where
\begin{equation}
  \label{eq:def-gagliardo-nirenberg-quadratic-form}
[v_1,v_2]_s = c_{N,s} \int_{\R^N \times \R^N}\frac{(v_1(x)-v_1(y))(v_2(x)-v_2(y))}{|x-y|^{N+2s}}dxdy
\end{equation}
and $[v]_s^2 := [v,v]_s$. In addition, we let $\cH^s(B)$ denote the closed subspace of functions $u \in H^s(\R^N)$ with $u \equiv 0$ in $\R^N \setminus B$. Furthermore, we assume that $\lambda >0$ in case of (\ref{eq:1.1RN}) and $\l>-\l_1(B)$ in case of (\ref{eq:1.1}), where $\l_1(B)$ denotes the first Dirichlet eigenvalue of $\Ds$ for $s\in (0,1)$. Finally, we assume that $p\in (1, 2^*_s-1)$ with  $2^*_s$ being the critical fractional Sobolev exponent given by
$$
  2^*_s=\frac{2N}{N-2s} \,\,\textrm{ if }  2s<N\qquad \textrm{ and } \qquad  2^*_s= \infty\,\, \textrm{ if } 2s\geq 1=N .
  $$
  A priori, it is natural to consider (\ref{eq:1.1RN}) and (\ref{eq:1.1}) in weak sense. A positive function $u \in H^s(\R^N)$ is called a weak solution of (\ref{eq:1.1RN}) if
  $$
  [u,v]_s + \lambda \int_{\R^N} u v\,dx = \int_{\R^N}u^{p}v \,dx \qquad \text{for all $v \in H^s(\R^N)$.}
  $$
  Similarly, a positive function $u \in \cH^s(B)$ is called 
a weak solution of (\ref{eq:1.1}) if
  $$
  [u,v]_s + \lambda \int_{B} u v\,dx = \int_{B}u^{p}v \,dx \qquad \text{for all $v \in \cH^s(B)$.}
  $$   
It is well known (see e.g. \cite[Proposition 3.1]{FLS}) that positive (weak) solutions to \eqref{eq:1.1RN} belong to $L^\infty(\R^N)\cap H^{s+1}(\R^N)\cap C^\infty(\R^N)$, are radially symmetric about a point $x_0\in \R^N$ and are strictly decreasing in $|x-x_0|$. We also recall that, by \cite{XJV}, every (weak) solution $u \in \cH^s(B)$ of (\ref{eq:1.1}) is contained in $L^\infty(\R^N)$, and therefore \cite{RS16a} and \cite{XJV} imply that $u\in C^s(\R^N)\cap C^\infty{(B)}$.  Moreover, by \cite{JW}, $u$ is a radial function which is strictly decreasing in its radial variable (see  also   \cite{BLW} for a different argument but with $\l=0$). 

In the local case $s=1$, the uniqueness of solutions to problems \eqref{eq:1.1} and  \eqref{eq:1.1RN} (up to translations) is a classical result and has been established by Kwong in his seminal paper \cite{Kwong}, see also \cite{Mc,C}.
In contrast, for semilinear fractional equations, the uniqueness and nondegeneracy of specific classes of solutions is a very challenging problem, and few results are available up to now. One of the main reasons for this is the lack of ODE techniques which are highly useful in the case $s=1$. 
It was first proved by Amick and Toland in \cite{AT}, based on complex analytic methods, that \eqref{eq:1.1RN} has a unique solution in the special case $s=1/2$, $p=2$ and $N=1$, see also \cite{Al} for an alternative proof. Much later, Frank, Lenzmann and Silvestre \cite{FLS} proved uniqueness and nondegeneracy of Morse index one solutions to \eqref{eq:1.1RN} up to translations in the general case, extending an earlier result of Frank and Lenzmann for the case $N=1$ to all dimensions, and extending also results of \cite{FV,KMR} to all $s\in (0,1)$. However, it remains open up to now whether \eqref{eq:1.1RN} may admit higher Morse index solutions. Moreover, in the case of (\ref{eq:1.1}), a uniqueness result is not even available for Morse index one solutions up to now.

The first purpose of the present paper is to show that all positive solutions of (\ref{eq:1.1}) are nondegenerate in nonradial directions, and that the same is true for (\ref{eq:1.1RN}) up to the directional derivatives of $u$. In order to formulate these statements rigorously, we consider the subspaces
$$
H^s_{r}(\R^N):= \{w \in H^s(\R^N)\::\: \text{$w$ radially symmetric}\}
$$
and 
$$
H^s_{nr}(\R^N):= \bigl[H^s_{r}(\R^N)\bigr]^\perp := \{u \in H^s(\R^N)\::\: \text{$\langle u, w \rangle_s = 0$ for all $w \in H^s_r(\R^N)$.}\}.
$$
Moreover, we let $\cH^s_{r}(B)$ resp. $\cH^s_{nr}(B)$ denote the subspaces of
$H^s_r(\R^N)$ and $H^s_{nr}(\R^N)$, respectively of functions $w$ with $w \equiv 0$ in $\R^N \setminus B$. We then have the following result.

\begin{theorem}\label{th:nondeg}
  Let $s\in (0,1)$, $N\geq 1$ and $1<p<2^*_s-1$.
  \begin{itemize}
  \item[(i)] Let  $u\in \cH^s(B)$ satisfy (\ref{eq:1.1}).
Then there exists $\ov \L>p$ such that 
   \begin{equation}
     \label{eq:B-th-eq}
 [w]_{s}^2 + \lambda \int_{B} w^2dx \geq \ov \L  \int_{\R^N} u^{p-1} w^2 dx  \qquad\textrm{ for all $w\in\cH^s_{nr}(B)$.} 
\end{equation}
 \item[(ii)] Let $u \in H^s(\R^N)$ satisfy \eqref{eq:1.1RN}. Then there exists $\ov \L>p$ such that 
   \begin{equation}
     \label{eq:r-n-th-eq}
 [w]_{s}^2+ \lambda  \int_{\R^N} w^2dx \geq \ov \L  \int_{\R^N} u^{p-1} w^2  dx  \qquad\textrm{ for all $w\in\cM_u$ ,} 
   \end{equation}
where
$$
\cM_u:= \Bigl\{w \in H^s_{nr}(\R^N)\,: \int_{\R^N}u^{p-1}(\de_{x_i}u) w\,dx=0 \quad\textrm{for all $i=1,\dots,N$}  \Bigr\}.
$$
  \end{itemize}
\end{theorem}

Let us explain in detail why Theorem~\ref{th:nondeg} should be regarded as a result on nonradial nondegeneracy. For this, let $u\in \cH^s(B)$ be a fixed radial solution of (\ref{eq:1.1}). The linearization of (\ref{eq:1.1RN}) at $u$ is then given by
the equation  
\begin{equation}
  \label{eq:linearization-B}
\Ds w+\l w-p u^{p-1}w= 0 \qquad \textrm{ in $B$,}\qquad \qquad  w\in \cH^s(B),
\end{equation}
which is closely related to the weighted eigenvalue problem 
\be \label{eq:weighted-eigen}
\Ds w+\l w=\L u^{p-1}w\qquad \textrm{ in $B$,}\qquad \qquad w\in \cH^s(B)
\ee
with eigenvalue parameter $\L\in \R$. In particular, if $\Lambda=p$ is not an eigenvalue of \eqref{eq:weighted-eigen}, then (\ref{eq:linearization-B}) has no nontrivial solution and $u$ is nondegenerate. We note that the compact embedding $\cH^s(B) \hookrightarrow L^2(B;u^{p-1}dx)$ implies that the set of eigenvalues of \eqref{eq:weighted-eigen} is given as a sequence $0<\L_1 <\L_2 \leq \dots$, with each eigenvalue being counted with multiplicity. Since $\Ds$ commutes with rotations, eigenfunctions corresponding to the eigenvalues $\L_i$ can be chosen to be contained in one of the spaces $\cH^s_r(B)$ and $\cH^s_{nr}(B)$. Moreover, we note that $\L_1$ is a simple eigenvalue with an eigenspace spanned by a positive radial eigenfunction. All other eigenfunctions of \eqref{eq:weighted-eigen} are orthogonal to the first eigenfunction with respect to the scalar product in $L^2(B;u^{p-1}dx)$, and therefore they have to change sign. Since $u$ solves \eqref{eq:weighted-eigen} with $\L=1$ and is positive, we conclude a posteriori that $\Lambda_1 = 1$ is the first eigenvalue of \eqref{eq:weighted-eigen}.

It is now natural to call $u$ {\em radially} resp. {\em nonradially nondegenerate} if (\ref{eq:linearization-B}) does not admit nontrivial solutions in $\cH^s_r(B)$, $\cH^s_{nr}(B)$, respectively. In particular, $u$ is {\em nonradially nondegenerate} if $\L = p$ is not an eigenvalue of \eqref{eq:weighted-eigen} admitting a nonradial eigenfunction\footnote{ We note here that a nonradial eigenfunction has a nontrivial $\langle \cdot, \cdot \rangle_s$-orthogonal projection on $\cH^s_{nr}(B)$ which is also an eigenfunction of \eqref{eq:weighted-eigen}}. Moreover, using again the compactness of the embedding $\cH^s(B) \hookrightarrow L^2(B;u^{p-1}dx)$, it is easy to see that the lowest eigenvalue $\Lambda_{nr}(u,B)$ of \eqref{eq:weighted-eigen} associated with nonradial eigenfunctions admits the variational characterization 
\begin{equation}
  \label{eq:var-char-l-2-ball-case}
\L_{nr}(u,B):=\inf \Bigl\{[w]_s^2 +{\l} \int_B w^2dx  \::\: w\in \cH^s_{nr}(B),\: \int_{B} w^2 u^{p-1}dx= 1 \Bigr\},
\end{equation}
and Theorem~\ref{th:nondeg}(i) is equivalent to the inequality $\L_{nr}(u,B) >p$.

Similar considerations apply to fixed radial solutions $u\in H^s(\R^N)$ of (\ref{eq:1.1RN}). We then consider the weighted  eigenvalue problem 
\be \label{eq:weighted-eigen-R^N}
\Ds w+ w=\L u^{p-1}w\qquad \textrm{ in $\R^N$,}\qquad \qquad  w\in H^s(\R^N)
\ee
with the positive radial and bounded weight function $u^{p-1}$. Since $u^{p-1}(x) \to 0$ as $|x| \to \infty$, the embedding $H^s(\R^N) \hookrightarrow L^2(\R^N,u^{p-1}dx)$ is compact. Therefore problem \eqref{eq:weighted-eigen-R^N} also admits a sequence of eigenvalues $0<\L_1<\L_2\leq \dots$, whereas $\Lambda_1 = 1$ is the first eigenvalue and $u$ is a corresponding eigenfunction.
Moreover, the linearized equation 
\begin{equation}
  \label{eq:linearization-r-n}
\Ds w+ w-p u^{p-1}w= 0 \qquad \textrm{ in $\R^N$,}\qquad \qquad  w\in H^s(\R^N),
\end{equation}
admits solutions of the form $v:=\n  u\cdot  \nu$ with $\nu \in \R^N$. Indeed, it follows from Lemma~\ref{lem:qual-sol} below that these functions are contained in the space $C^\infty(\R^N)\cap H^s_{nr}(\R^N)$ and are pointwise solutions of (\ref{eq:linearization-r-n}). From a more geometric point of view, this degeneracy is due to the translation invariance of equation~(\ref{eq:1.1RN}). As a consequence of this nonradial degeneracy, the lowest eigenvalue of \eqref{eq:weighted-eigen-R^N} associated with nonradial eigenfunctions 
\begin{equation}
  \label{eq:var-char-l-2-rn-case}
\L_{nr}(u,\R^N):=\inf \Bigl\{[w]_s^2 +{\l} \int_{\R^N} u^2dx  \::\: w\in H^s_{nr}(\R^N),\: \int_{\R} w^2 u^{p-1}dx= 1 \Bigr\},
\end{equation}
satisfies the inequality $\Lambda_{nr}(u,\R^N) \le p$. Theorem~\ref{th:nondeg}(ii) now implies that $\Lambda_{nr}(u,\R^N) = p$ and that the derivatives $v = \n  u\cdot  \nu$, $\nu \in \R^N$ are the only nonradial solutions of (\ref{eq:linearization-r-n}). Similarly as for (\ref{eq:1.1}),  a solution $u$ to \eqref{eq:1.1RN} is called \textit{radially nondegenerate} if  \eqref{eq:linearization-r-n} does not admit a nontrivial radial solution.

Our proof of Theorem~\ref{th:nondeg} is based on a new Picone type identity, which we shall apply to eigenfunctions of \eqref{eq:weighted-eigen} and of \eqref{eq:weighted-eigen-R^N} which are antisymmetric with respect to the reflection at a hyperplane containing the origin. In the case of the ball, this identity directly rules out the existence of antisymmetric eigenfunctions corresponding to eigenvalues $\ov \L \le p$, and this is sufficient since any eigenfunction can be written as a sum of its symmetric and antisymmetric part with respect to a hyperplane reflection, and the antisymmetric part must be nonzero for at least one hyperplane containing the origin if the eigenfunction is nonradial. We note here that, very recently, a different proof of Theorem~\ref{th:nondeg}(i), based on polarization, has been given independently in the revised version of \cite{DIS-1}. 
The proof of Theorem~\ref{th:nondeg}(ii) is more involved, and we need to combine the Picone identity with a new geometric lemma in order to characterize the directional derivatives $\n  u\cdot  \nu $ as the only functions in $H^s_{nr}(\R^N)$ which solve (\ref{eq:linearization-r-n}).    We note that the argument used to prove  Theorem~\ref{th:nondeg}(ii)  is purely nonlocal and therefore can be generalized to   a larger class of nonlocal operators, see Remark \ref{lem:rema-gen-non} below.\\
Theorem~\ref{th:nondeg} leaves open the problem of radial (and therefore full) nondegeneracy and uniqueness of positive solutions to (\ref{eq:1.1RN}) and (\ref{eq:1.1}). As remarked above, these properties have already been established for Morse index one solutions of (\ref{eq:1.1RN}), whereas very little is known for (\ref{eq:1.1}) and for the more general class of all positive solutions. In fact, in addition to the very special result of Amick and Toland \cite{AT} for the case $s=1/2$, $p=2$ and $N=1$, we are only aware of the very recent paper \cite{DIS} which provides interesting and useful uniqueness results within borderline ranges of parameters $s$ and $p$, assuming more precisely that $s$ is close to $1$ or $p$ is close to $1$ or $2^{*}_s-1$.

Our second main result of the paper seems to be the first result which is valid for all $p$ and $s$, but we need to restrict to the ball case in space dimension $N=1$ and to  $\lambda=0$.\footnote{In a previous version of the paper, we stated uniqueness of positive solutions of (\ref{eq:1.1RN}) and (\ref{eq:1.1}) for the full range of parameters $N$, $\lambda$, $s$ and $p$. However, our proof contained a sign mistake within a polarization argument to exclude the existence of radial eigenfunctions corresponding to the second eigenvalue $\Lambda_2$ in \ref{eq:weighted-eigen} and \ref{eq:weighted-eigen-R^N}. Essentially the same mistake has been made in earlier versions of \cite{Azahara-Parini} and \cite{DIS-1}.} 

\begin{theorem}\label{new-uniqueness-intervall}
  Let $s\in (0,1)$, $N= 1$, $1<p<2^*_s-1$ and $\lambda=0$, so that (\ref{eq:1.1}) writes as
  \begin{equation}
    \label{eq:1.1-simplified}
\Ds u = u^{p}\quad\text{ in ${B}=(-1,1)$},\qquad u>0 \quad\textrm{in ${B}$}, \quad u \in \cH^s(B)
  \end{equation}
  Then the problem (\ref{eq:1.1-simplified}) has a unique   solution which is fully nondegenerate in the sense that $\Lambda_2>p$ for the second eigenvalue $\Lambda_2$ of the weighted eigenvalue problem
\be \label{eq:weighted-eigen-simplified}
\Ds w =\L u^{p-1}w\qquad \textrm{ in $B$,}\qquad \textrm{$w\in \cH^s(B)$. }
\ee
\end{theorem}
Our next result   is devoted to  nondegeneracy and uniqueness of even solutions to the problems \eqref{eq:1.1RN} and \eqref{eq:1.1}   in the special case $p=2$ and $s>\frac{1}{6}$.  We note the lower bound on $s$ comes from  the fact that, for $N=1$, we have 
 $$ 
2<2^*_s-1= \frac{1+2s}{(1-2s)_+} \qquad\Longleftrightarrow\qquad  s>\frac{1}{6}.
 $$
Up to date,  the only \textit{general} uniqueness result,  up to translation,  of problem  (\ref{eq:1.1RN}) is due to Amick and Toland \cite{AT}, and it is concerned with  the special case $N=1$, $s=1/2$ and $p=2$.  We extend their result to the full subcritical range of the  parameter $s>1/6$. 
\begin{theorem}\label{new-uniqueness-AT}
  Let $N=1$,  $s\in (\frac{1}{6},1)$, $p=2$ and $\l>0$.
  Then each of the the problems \eqref{eq:1.1RN} and (\ref{eq:1.1})  admits a unique even and radially nondegenerate   solution.
\end{theorem}
We  observe that  the lower bound $\frac{1}{6}$ is optimal since \eqref{eq:1.1RN} and (\ref{eq:1.1})  do not have a  solution if $2\geq 2^*_s-1$ as a consequence of the fractional Pohozaev identity (see \cite{RX-Poh}).
\begin{remark}
We recall that \cite{AT} proved also a uniqueness property of nonconstant positive periodic solutions to $\Ds u+u=u^2$ on $\R$, for $s=\frac{1}{2}$,  which corresponds to periodic solitary waves of the Benjamin-Ono equation. We believe that,  for values of  $s\not=\frac{1}{2}$,    the arguments in the present paper can be also used to prove uniqueness of even   nonconstant $T$-periodic solutions which are decreasing on $(0,\frac{T}{2})$.   We refer to \cite{CCM} which deals with the existence of such type of solutions for more general nonlocal problems.  See also \cite{BGQ,Sa} and the references therein  for the existence of nonconstant  periodic solution to this problem.
\end{remark}

The proofs of  Theorem \ref{new-uniqueness-intervall} and Theorem~\ref{new-uniqueness-AT} combine various tools. The first step is to prove the radial nondegeneracy of every positive solution of (\ref{eq:1.1-simplified}), and the second step is a continuation argument based on varying the exponent $p$ and using the already available asymptotic uniqueness result given in \cite[Theorem 1.6]{DIS} for the case $p>1$. In the proof of the radial nondegeneracy of solutions to (\ref{eq:1.1-simplified}), which is stated in Theorem~\ref{radial-nondegeneracy-combined} below, we shall first use a fractional integration by parts formula to show that every radial eigenfunction $w$ of \eqref{eq:weighted-eigen-simplified} corresponding to the eigenvalue $\L = p$  must have a vanishing fractional boundary derivative. From this property, we then derive information on the shape of nodal domains of the $s$-harmonic extension (also called Caffarelli-Silvestre extension) $W$ of $w$. Using this information, we then construct an odd function in $\cH^s(B) \setminus \{0\}$ for which equality holds in \eqref{eq:B-th-eq} with $\lambda = 0$, $\ov \L\le p$, contradicting Theorem~\ref{th:nondeg}. We conjecture that a similar strategy might work for the more general problems (\ref{eq:1.1RN}) and (\ref{eq:1.1}) in the case $N=1$, but this remains open unless for $p=2$, where special cancellations occur and allow to complete the proof of Theorem~\ref{new-uniqueness-AT}. We stress here that our proof of Theorem~\ref{new-uniqueness-AT} does not use complex analysis and is therefore quite different from the argument of Amick and Toland in \cite{AT} which only applies to the special case $s=\frac{1}{2}$. 

By similar methods as in the proof of Theorem~\ref{new-uniqueness-intervall}, we also prove the following result  on uniqueness and nondegeneracy of ground state solutions to   (\ref{eq:1.1}) in the special case $N=1$ and $\l>0$.  Here, following the notation of \cite{FLS}, a ground state solution $u$ is defined as a Morse index one solution, i.e., a solution $u$ with the property that $\Lambda_2 \ge p$ for the second eigenvalue $\Lambda_2$ of \eqref{eq:weighted-eigen}. We note that this class of solutions include least energy solutions of (\ref{eq:1.1}). The result parallels those of \cite{FLS} for (\ref{eq:1.1RN}) in the special case $N=1$.

\begin{theorem}\label{new-uniqueness-intervall-GS}
  Let $s\in (0,1)$, $N= 1$, $1<p<2^*_s-1$ and $\lambda>0$, so that (\ref{eq:1.1}) writes as
  \begin{equation}
    \label{eq:1.1-simplified-GS}
\Ds u+\l u = u^{p}\quad\text{ in ${B}=(-1,1)$},\qquad u>0 \quad\textrm{in ${B}$}, \quad u = 0\quad\text{in $\R \setminus {B}$.}    
  \end{equation}
  Then the problem (\ref{eq:1.1-simplified-GS}) has a unique  ground state solution solution which is fully nondegenerate in the sense that $\Lambda_2>p$ for the second eigenvalue $\Lambda_2$ of the weighted eigenvalue problem
\be \label{eq:weighted-eigen-simplified-GS}
\Ds w +\l w =\L u^{p-1}w\qquad \textrm{ in $B$,}\qquad \textrm{$w\in \cH^s(B)$. }
\ee
\end{theorem}

As mentioned above, the proof of Theorem~\ref{new-uniqueness-intervall-GS} is merely a variant of the proof of Theorem~\ref{new-uniqueness-intervall}, as it also relies on the two-dimensional nature of the extension problem related to \ref{eq:weighted-eigen-simplified-GS}. Theorem~\ref{new-uniqueness-intervall-GS} leaves open the uniqueness and nondegeneracy of positive {\em ground state solutions} of (\ref{eq:1.1}) in the higher dimensional unit ball $B \subset \R^N$, $N \ge 2$. By a rather different and much more involved argument based on a new characterization of sign properties of second radial Dirichlet eigenfunctions of fractional Schrödinger type operators in the unit ball, we prove the  analogue of Theorem~\ref{new-uniqueness-intervall-GS} for all dimensions $N \ge 1$ and all $\lambda \ge 0$ in the recent paper \cite{fall-weth-m1-2024}. We point out that there is no overlap of the present paper with \cite{fall-weth-m1-2024}. In particular, the
arguments in \cite{fall-weth-m1-2024} neither yield Theorem~\ref{new-uniqueness-intervall} nor Theorem~\ref{new-uniqueness-AT}. We also point out that, in the special case $\lambda=0$, uniqueness and nondegeneracy of positive ground state solutions of (\ref{eq:1.1}) in the higher dimensional case has been proved very recently and independently from our work in the most recent version of \cite{Azahara-Parini}.

For the final result of the paper, we go back to the full space problem (\ref{eq:1.1RN}) in the special case $s = \frac{1}{2}$ and $p = 1 + \frac{2}{N+1}$. In this case, we can compute the unique ground state solution of (\ref{eq:1.1RN}) up to translation and deduce a sharp fractional Gagliardo-Nirenberg inequality from this computation. More precisely, we have the following.

\begin{theorem}
  \label{new-gagliardo-nirenberg-uniqueness}
Let $p= 1+ \frac{2}{N+1}$. Then every ground state solution of the equation
\begin{align}\label{eq:1.1RN-special-case}
(-\Delta)^{\frac{1}{2}} u+ u  = u^{p}\quad\text{ in $\R^N$},\qquad u>0 \quad\textrm{ in $\R^N$}, \qquad  u\in H^{1/2}(\R^N)
\end{align} 
is of the form $u_a(x)= u(x-a)$, $a \in \R^N$ with
\begin{equation}
  \label{eq:u-explicit-intro}
u(x)= \Bigl(\frac{N+1}{1+|x|^2}\Bigr)^{\frac{N+1}{2}}.
\end{equation}
Moreover, the optimal constants $C_S, C_{GN}>0$ in the associated fractional Sobolev type and Gagliardo-Nirenberg inequalities 
\begin{equation}
  \label{eq:Sobolev-type}
\|v\|_{L^{q}(\R^N)} \le C_S \|v\|_s, \qquad v \in H^{1/2}(\R^N) 
\end{equation}
and 
\begin{equation}
  \label{eq:GN-type}
\|v\|_{L^{q}(\R^N)} \le C_{GN} [v]_s^{\frac{2}{N+2}}  \|u\|_{L^2(\R^N)}^{\frac{2}{N+2}},\qquad v \in H^{1/2}(\R^N) 
\end{equation}
with $q=p+1= \frac{2(N+2)}{N+1}$ are given by
$C_S =(N+1)^{-\frac{N+1}{2(N+2)}} \Bigl( \frac{(N+2) \pi^{N/2}\Gamma(\frac{N}{2}) }{4 \Gamma(N)}  \Bigr)^{-\frac{1}{2(N+2)}}$
 and $C_{GN} = \Bigl(\frac{N^N (N+1)^{N+1} \pi^{\frac{N}{2}}\Gamma(\frac{N}{2})}{(N+2)^{N+1}\:\Gamma(N)}\Bigr)^{-\frac{1}{2(N+2)}}.$
\end{theorem}
Note that if  $N=1$, then $1+ \frac{2}{N+1}=2$,   so that  \eqref{eq:1.1RN-special-case} corresponds to the Benjamin-Ono equation,  and  the explicit solution given by \eqref{eq:u-explicit-intro} was found by Benjamin \cite{Ben}.\\
The fact that $u$ given in (\ref{eq:u-explicit-intro}) solves (\ref{eq:1.1RN-special-case}) will be shown with a Fourier transform argument. To show that $u$ is a ground state solution, we apply Theorem~\ref{th:nondeg} and an argument inspired by \cite[Section 7.1]{FLS}. The uniqueness up to translations within the class of ground state solutions then follows by \cite[Theorem 4]{FLS}. Since the best constant $C_S$ in (\ref{eq:Sobolev-type}) is attained for $v=u$, it can then be computed easily. Moreover, the constants $C_S$ and $C_{GN}$ are related by a general principle (see Lemma~\ref{lem-Sobolev-Gagliardo-Nirenberg} below), and this allows to compute $C_{GN}$. 

The paper is organized as follows.
In Section~\ref{sec:preliminaries} we first derive preliminary information on the shape and regularity of positive solutions of (\ref{eq:1.1RN}) and (\ref{eq:1.1}). In Section \ref{sec:nondegeneracy}, we then derive a Picone type identity for antisymmetric functions, which is a key tool in the analysis of nonradial solutions of the linearized problems associated with (\ref{eq:1.1RN}) and (\ref{eq:1.1}).
In Section~\ref{sec:nonr-nond}, we then complete the proofs of the nonradial nondegeneracy results given in Theorems~\ref{th:nondeg}.
In Section~\ref{sec:radi-nond-spec}, we prove the radial nondegeneracy of positive solutions of (\ref{eq:1.1-simplified}), of positive ground state solutions of (\ref{eq:1.1-simplified-GS}) and of positive even solutions of (\ref{eq:1.1RN}) and (\ref{eq:1.1}) in the case $N=1$, $p=2$ and $s > \frac{1}{6}$. Section~\ref{sec:uniqueness-case-ball} is then devoted to the continuation argument which completes the proofs of Theorems~\ref{new-uniqueness-intervall},~\ref{new-uniqueness-AT} and~\ref{new-uniqueness-intervall-GS}. Section~\ref{special-case} contains the proof of Theorem~\ref{new-gagliardo-nirenberg-uniqueness}. Finally, in the appendix, we first provide a geometric lemma which is also used in the analysis of nonradial eigenfunctions.  Moreover, we provide a variant of the fractional Hopf type lemma in \cite[Prop. 2.2]{SV} with somewhat weaker assumptions, which is essential in our context. Finally, we state a topological lemma on curve intersection which is used in the analysis of nodal domains of the $s$-harmonic extensions of eigenfunctions of \eqref{eq:weighted-eigen-simplified} and \eqref{eq:weighted-eigen-simplified-GS}.

\section{Preliminaries}
\label{sec:preliminaries}

In this section, we collect some preliminary information on solutions of (\ref{eq:1.1RN}) and (\ref{eq:1.1}).

\begin{lemma}\label{lem:qual-sol}
  \begin{itemize}
  \item[(i)] Let $u \in \cH^s(B)$ satisfy (\ref{eq:1.1}). Then $u\in C^\infty(B)\cap C^s(\R^N)$, and $u$ is radially symmetric and  strictly decreasing. Moreover,
    \begin{equation}
      \label{eq:qual-sol-eq-1}
    \inf_{x\in B}\frac{u(x)}{(1-|x|)^s}>0
    \end{equation}
    and
    \begin{equation}
      \label{eq:qual-sol-eq-2}
    \lim_{|x| \nearrow 1} (1-|x|)^{1-s}\n u(x)\cdot x <0.
    \end{equation}
  \item[(ii)]   Let $u\in H^s(\R^N)$ satisfy \eqref{eq:1.1RN}.    Then 
 $u\in   C^\infty(\R^N)\cap H^{1+2s}(\R^N)$. Moreover $u$ is radially symmetric about a point $x_0\in \R^N$, strictly decreasing  as a function of $|x-x_0|$ and satisfies $u(x) \to 0$ as $|x| \to \infty$.
  \end{itemize}
\end{lemma}

\begin{proof}
(ii) has been observed in \cite{FLS}.\\
(i) From  \cite[Proposition 3.1]{XJV} we  have that $u\in L^\infty(B)$. Hence by a classical bootstrap argument combined with interior and boundary  regularity, in \cite{RS16a}, we find that $u\in C^s(\R^N)\cap C^\infty(B)$.  Thanks to \cite[Corollary 1.2]{JW} we deduce that $u$ is radially symmetric   and  strictly decreasing in $B$. Now (\ref{eq:qual-sol-eq-1}) follows from the fractional Hopf lemma,  see e.g. \cite[Proposition 3.3]{FJ-2015}. Finally, by \cite{FJ-2019} we have
$$
\lim_{|x| \nearrow 1} (1-|x|)^{1-s}\n u(x)\cdot x=-s \lim_{|x| \nearrow 1}\frac{u(x)}{(1-|x|)^s},
$$
and (\ref{eq:qual-sol-eq-2}) follows.
\end{proof}

\begin{remark}
\label{rem-fractional-derivative}  
Let $f \in L^\infty(B)$, and let $w \in \cH^s(B)$ satisfy
$$
\Ds w = f  \qquad \text{on $B$}
$$
in weak sense, i.e.
$$
[w,v]_s = \int_{B}f v\,dx \qquad \text{for all $v \in \cH^s(B)$.}
$$
Then, by the fractional elliptic regularity theory  regularity theory in \cite{RS16a}, the function 
\begin{equation}
  \label{eq:def-fractional-normal-deriv}
\psi_w: B \to \R, \qquad \psi_w(x):= \frac{w(x)}{(1-|x|)^s}= \frac{w(x)}{\dist(x,\partial B)^s}
\end{equation}
extends to a (Hölder) continuous function on $\overline B$. The restriction of $\psi_w$ to $\partial B$ is called the {\em fractional (inner) normal derivative}.
It follows from Lemma~\ref{lem:qual-sol}(i) that $\psi_u >0$ on $\partial B$ if $u \in \cH^s(B)$ satisfies (\ref{eq:1.1}). Since $u$ is radial in this case, we may simply write $\psi_u(1)>0$, so we write $\psi_u$ as a function of the radial variable.
\end{remark}

\section{A Picone type identity}
\label{sec:nondegeneracy}

%


The following new Picone-type result will be one of our main tools in the following.
\begin{lemma}\label{lem:Picone}
Let $\O$ be an open set of $\R^N$ and $\a>0$.  Let  $V\in L^1_{loc}(\O)$ and    $v\in C^{2s+\a}(\O)\cap L^1(\R^N;(1+|x|)^{-N-2s})$ satisfy
\be\label{eq:identi-inveigen-A} 
\Ds v= V v \qquad\textrm{ in $\O$} .
\ee
Let $e$ be a unit vector of $\R^N$ and let $\s_e$ denote the   reflection with respect to $\{x\in \R^N\,:\,x\cdot e=0\}$.    
Suppose that $v>0$ on $\{x\in \R^N\,:\,x\cdot e>0\}$ and $v=-v\circ \s_e$ on $\R^N$.
Then for any $w\in  H^s(\R^N)\cap C_c(\O)$  satisfying $w=-w\circ \s_e$ on $\R^N$ and $\frac{w}{v}\in C_c(\O)$,  we have 
\be\label{eq:Poincare}
[w]_{s}^2 - \int_{\R^N} V w^2dx=\int_{\{x\cdot e>0\}} \int_{\{y\cdot e>0\}} H_{w,v}^e(x,y)dxdy,
\ee
where  $H_{w,v}^e\geq 0$ on  $\{x\cdot e>0\}\cap \{y\cdot e>0\}$ is given by 
\be \label{eq:Hewv}
H_{w,v}^e(x,y):=c_{N,s} {v(x)v(y)}  \left[ w(x)/v(x)- w(y)/v(y)\right]^2 
\times \left( \frac{1}{|x-y|^{N+2s}}- \frac{1}{ |\s_e( x) -y|^{N+2s}} \right) .
\ee
\end{lemma}
\begin{proof}
We assume without loss of generality that $e=e_1$. We first note that for every $x,y\in \R^N$ with $x_1\not=0$ and $y_1\not=0$,  we have 
\begin{align}\label{eq:identi-inveigen}
(w(x)-w(y))^2-&w^2(x)\frac{(v(x)-v(y) )}{v(x)}  + w^2(y)\frac{(v(x)-v(y) )}{v(y)} \nonumber\\
&=v(x)v(y)\left( w(x)/v(x)- w(y)/v(y)\right)^2.
\end{align}
Moreover, by assumption, $w/v\in C_c(\O)$, and thus 
$$
c_{N,s}\int_{\R^N\times \R^N} \frac{w^2(x)}{v(x)}\frac{v(x)-v(y) }{|x-y|^{N+2s}} dxdy= \int_{\O} \frac{w^2(x)}{v(x)}\,\Ds v(x) dx
=\int_{\O}V(x) w^2(x) dx<\infty
$$
by \eqref{eq:identi-inveigen-A}. From this  and a change of variable,  we  get 
\begin{align*}
&\int_{\R^N\times \R^N}  \frac{(w(x)-w(y))^2}{|x-y|^{N+2s}}   dy dx-\frac{2}{c_{N,s}} \int_{\O}V(x) w^2(x) dx \\
&= \int_{\R^N\times \R^N} \frac{v(x)v(y)}{ |x-y|^{N+2s}} \left( w(x)/v(x)- w(y)/v(y)\right)^2  dx  dy\\
&=2 \int_{\{x_1>0\}\cap \{y_1>0\}} \frac{v(x)v(y)}{|x-y|^{N+2s}} \left( w(x)/v(x)- w(y)/v(y)\right)^2   dx  dy\\
&-2 \int_{\{x_1>0\}\cap \{y_1>0\}} \frac{v(x)v(y)}{((x_1+y_1)^2+|\ti x -\ti y|^2)^{(N+2s)/2}} \left( w(x)/v(x)- w(y)/v(y)\right)^2  dx  dy.
\end{align*}
Here, we used that $x_1\mapsto v(x_1,\cdot)$ is odd and $x_1\mapsto \frac{ w}{v}(x_1,\cdot)$ is even.  It follows,  from this,  that   
\begin{align}
&\int_{\R^N\times \R^N}  \frac{(w(x)-w(y))^2}{|x-y|^{N+2s}}   dy dx -\frac{2}{c_{N,s}} \int_{\O}Vw^2 dx  =\frac{2}{c_{N,s}}  \int_{\{x_1>0\}\cap \{y_1>0\}}H^e_{w,v}(x,y)   dx  dy, \label{eq:nearfin-A}
\end{align}
where, for all $(x,y)\in \{x_1>0\}\times \{y_1>0\}$,
\begin{align*}
H^e_{w,v}(x,y):=c_{N,s} {v(x)v(y)} & \left[ w(x)/v(x)- w(y)/v(y)\right]^2\\
&\times \left( |x-y|^{-N-2s}- ((x_1+y_1)^2+|\ti x -\ti y|^2)^{-(N+2s)/2} \right)  .
\end{align*}

By the triangular inequality and the fact that $v(x)v(y)>0$ for all $x,y\in \{x_1>0\}\cap \{y_1>0\}$,  we deduce that $H^e_{w,v}\geq 0$ on  $\{x_1>0\}\cap \{y_1>0\}$. Hence \eqref{eq:Poincare} follows.
\end{proof}

\section{Nonradial nondegeneracy}
\label{sec:nonr-nond}

\subsection{Nonradial nondegeneracy in the ball case}
\label{sec:nonr-nond-ball}
This section is devoted to the proof of Part (i) of Theorem~\ref{th:nondeg}. By the remarks after 
Theorem~\ref{th:nondeg}, it suffices to prove the following.

\begin{theorem}
\label{r-n-case-nonradial-nond-section-ball-case}  
Let $w \in \cH^s_{nr}(B)$ be an eigenfunction of \eqref{eq:weighted-eigen} corresponding to an eigenvalue $\Lambda \in \R$.
Then $\Lambda>p$.
\end{theorem}

\begin{proof}
  Assume by contradiction that $\Lambda \leq p$. For a unit vector $\nu$, we define the hyperplane $T_\nu:=\{x\cdot\nu=0\}$ and we denote by $\s_\nu $ the reflection with respect to $T_\nu $. Since $w$ is nonradial, there exists a unit vector $\nu$ with the property that 
$$
{\ti w}:=\frac{ w- w\circ\s_\nu}{2} \in \cH^s(B) \setminus \{0\}. 
$$
Moreover, since $u$ is radial,  ${\ti w} $ also solves $\Ds {\ti w}+{\ti w}=\Lambda u^{p-1} {\ti w}$ in $B$.  Without loss of generality, we may assume that $\nu = e_1$, so ${\ti w}$ is odd with respect to reflection of the $x_1$-variable. By Lemma \ref{lem:qual-sol}(i),  we have  $v:=-\de_{x_1}u \in C^\infty(B)\cap L^1(\R^N)$, and $v$ is a pointwise solution 
of the equation $\Ds v+\l v=pu^{p-1} v$ in $B$ which is odd with respect to the reflection at $\{x_1>0\}$ and satisfies
$$
v \ge 0\quad \text{in $\{x_1>0\}$}\qquad \text{and}\qquad v \not \equiv 0 \quad \text{in $\{x_1>0\}\cap B$.}
$$
Indeed, these inequalities follow since $v(x)=-{x_1}\frac{\de_r u(x)}{|x|}$ and $\de_r u \le 0$, $\de_r u \not \equiv 0$ in $B \setminus \{0\}$.

Now, by Lemma~\ref{hopf-type-lemma} from the appendix, we have $v>0$ in $\{x_1>0\}\cap B$ and $\partial_{x_1}v >0$ on $\{x_1=0\} \cap B$. Consequently, 
\begin{equation}
  \label{eq:quotient-x-1}
x \mapsto \frac{v(x)}{x_1} \qquad \text{extends to a positive $C^\infty$-function on $B$.}  
\end{equation}
Next we note that, by fractional elliptic regularity, ${\ti w} \in C^\infty(B)\cap L^\infty(B)$.
For $k \in \N$, we define the functions
\begin{equation}
  \label{eq:def-z-k-etc}
  \z_k(x)=1- \chi(k (1-|x|) ),
\end{equation}
  where $\chi$ is given by  \eqref{eq:def-chi}, 
and we put  $w_k:=\z_k {\ti w} \in C^1_c(B)$.  Therefore $w_k/v\in C_c(B)$ by \eqref{eq:quotient-x-1}. We can thus   
apply Lemma \ref{lem:Picone} to  obtain 
\begin{align}
[w_k]^2+ \l\int_B w_k^2dx -p\int_{B}u^{p-1} w_k^2 dx 
= \int_{\{x_1>0\}\cap \{y_1>0\}}H_{w_k,v}^{e_1}(x,y)   dx  dy. \label{eq:nearfin}
\end{align}
By \cite[Lemma 2.2]{DFW}, we have 
$$
\int_{\R^N\times \R^N}  \frac{(w_k(x)-w_k(y))^2}{|x-y|^{N+2s}}   dy dx \to \int_{\R^N\times \R^N}  \frac{({\ti w}(x)-{\ti w} (y))^2}{|x-y|^{N+2s}}   dy dx \quad\textrm{as $k\to \infty$. }
$$
In addition, by the dominated convergence theorem,
$$
\int_{B}u^{p-1} w_k^2 dx\to \int_{B}u^{p-1} {\ti w}^2 dx\quad \text{and}\quad \int_{B} w_k^2 dx\to \int_{B}  {\ti w}^2 dx \qquad \text{as $k\to \infty$.}
$$
We can thus  apply Fatou's lemma in \eqref{eq:nearfin}  to deduce that 
 \begin{align}
0=  [{\ti w}]_s^2+\l\int_B {\ti w}^2dx-\Lambda\int_B u^{p-1}{\ti w}^2dx\geq  
   &[{\ti w}]_s^2+\l\int_B {\ti w}^2dx-p\int_B u^{p-1}{\ti w}^2dx \nonumber\\
   &\geq  \int_{\{x_1>0\}\cap \{y_1>0\}}  H^{e_1}_{{\ti w},v}(x,y)dxdy, \label{piccone-consequence-ball} 
 \end{align}
where, recalling \eqref{eq:Hewv},
\begin{align*}
& H^{e_1}_{\ti w,v}(x,y):=c_{N,s} {v(x)v(y)}  \left[ \ti w(x)/v(x)- \ti w(y)/v(y)\right]^2\\
&\times \left( |x-y|^{-N-2s}- ((x_1+y_1)^2+|\ti x -\ti y|^2)^{-(N+2s)/2} \right) \geq 0 \quad\textrm{for all $x,y\in \{x_1>0\}\cap \{y_1>0\}$ .} 
\end{align*}
From this, (\ref{piccone-consequence-ball}) and the assumption $\Lambda \le p$, we deduce $\Lambda = p$ and that equality holds in (\ref{piccone-consequence-ball}). This implies that $  H^{e_1}_{w,v}\equiv 0$ on ${\{x_1>0\}\cap \{y_1>0\}}$ and thus   $w\not\equiv 0$  is  proportional to $v=\de_{x_1}u$ on $\{x_1>0\}$.   This  is impossible because  $\de_{x_1}u$ is unbounded as $|x|\to 1$ on  $\{x_1>0\}$ by Lemma \ref{lem:qual-sol}(i).  We thus conclude that   $\Lambda > p$. 
\end{proof}

\subsection{Nonradial nondegeneracy  in the case of $\R^N$}
\label{sec:nonr-nond-case}

This section is devoted to the proof of Part (ii) of Theorem~\ref{th:nondeg}. Hence we assume from now on that $u=u(|x|)$ is a fixed radial solution of (\ref{eq:1.1RN}). By the remarks after Theorem~\ref{th:nondeg}, it suffices to prove the following. 
\begin{theorem}
\label{r-n-case-nonradial-nond-section}  
Let $w \in H^s_{nr}(\R^N)$ be an eigenfunction of \eqref{eq:weighted-eigen-R^N} corresponding to an eigenvalue $\Lambda \in \R$. Then
one of the following alternatives hold.
\begin{itemize}
\item[(i)] $\Lambda>p$.
\item[(ii)] $\Lambda=p$, and $w \in \spann \{\de_{x_1}u,\dots,\de_{x_N}u\}$.   
\end{itemize}
\end{theorem}

\begin{proof}
Suppose that $\Lambda \le p$.  By elliptic regularity we have $ w\in C^2(\R^N)$. We fix a unit vector $\nu$ and consider, as before, the reflection $\s_\nu $ at the hyperplane $T_\nu:=\{x\cdot\nu=0\}$. Moreover, we define
$$
{w^\nu}:=\frac{ w- w\circ\s_\nu}{2} \in H^s(\R^N)\cap C^2(\R^N).
$$
Clearly since $u$ is radial,  then  ${w^\nu} $  solves $\Ds {w^\nu}+{w^\nu}=\L_2 u^{p-1} {w^\nu}$ in $\R^N$.  
By Lemma~\ref{lem:qual-sol}(ii),  we have  $v:=-\n  u\cdot  \nu \in C^\infty(\R^N)\cap H^s(\R^N)$, and $v$ is a pointwise solution 
of the equation $\Ds v+ v=pu^{p-1} v$ in $\R^N$ which is odd with respect to the reflection at $\nu$ and $v  \gneqq0$ on  $\{x\cdot \nu>0\}$.  Applying  Lemma~\ref{hopf-type-lemma} from the appendix, we have $v>0$ in $\{x\cdot \nu >0\} $ and $\n v\cdot \nu >0$ on $\{x\cdot \nu=0\} $.  Therefore,
\begin{equation}
\label{eq:quotient-x-nu}
x \mapsto \frac{v(x)}{x\cdot \nu} \qquad \text{extends to a positive  $C^\infty$-function on $\R^N$.}  
\end{equation}
We let 
\be \label{eq:def-chi}
\chi\in C^\infty_c(-2,2),\quad\textrm{ with $0 \le \chi \le 1$ on $\R$ and $\chi\equiv 1$ on $(-1,1)$}.
\ee 
 For $R\in \N$,  we define $w_R^\nu(x)=w^\nu(x) \chi(|x|/R)$ for all $x\in \R^N$,  and we note that $w_R/v\in C_c(\R^N)$ thanks to \eqref{eq:quotient-x-nu}.  
Applying  Lemma \ref{lem:Picone},  we get 
\begin{align}
[w_R^\nu]_s^2+  \int_{\R^N} (w_R^\nu)^2dx - p \int_{\R^N}u^{p-1} (w_R^\nu)^2 dx  = \int_{\{x\cdot \nu >0\}\cap \{y \cdot \nu >0\}}H_{w_R^\nu ,v}^\nu(x,y)   dx  dy.\label{eq:nearfin-RN}
\end{align}
It is not difficult to see that $\|w_R^\nu\|_{H^s(\R^N)}\to \|w^\nu\|_{H^s(\R^N)}$ as $R\to \infty$.  On the other hand by the dominated convergence theorem,  as $R\to \infty $,
$$
\int_{\R^N} (w_R^\nu)^2dx\to  \int_{\R^N} (w^\nu)^2dx\qquad\textrm{and}\qquad   \int_{\R^N}u^{p-1} (w_R^\nu)^2 dx\to  \int_{\R^N}u^{p-1} (w^\nu)^2 dx.
$$
We thus apply  Fatou's lemma, recalling \eqref{eq:Hewv},   and  obtain that 
\be \label{eq:L2nearp}
 (\Lambda-p) \int_{\R^N} u^{p-1}(w^\nu)^2 dx\geq   \int_{\{x\cdot \nu>0\}\cap \{y\cdot \nu>0\}}H_{w^\nu,v}^\nu (x,y)   dx  dy\geq 0. 
\ee
Since $\Lambda \leq p$ by assumption, we conclude that $\Lambda= p$ and $H_{w^\nu,v}^\nu  \equiv 0$ on $\{x\cdot \nu>0\}\cap  \{y\cdot \nu>0\}$.  In view of \eqref{eq:Hewv}, this implies that there exists $\k_\nu \in \R$ such that   
$$
\frac{ w(x)- w\circ\s_\nu(x) }{2} = \k_\nu  \n u(x)\cdot \nu=\k_\nu \partial_r u(x) \frac{\nu\cdot x}{|x|}  \qquad\textrm{ for all $x\in \R^N$.}
$$
We recall that this holds for an arbitrary unit vector $\nu \in S^{N-1}$. Moreover, a posteriori we find that
$$
\k_\nu = \|\nabla u \cdot \nu \|_{L^2(\R^N)}^{-2} \int_{\R^N} w \nabla u \cdot \nu \,dx \qquad \text{for $\nu \in S^{N-1}$,}
$$
so the function $\nu \mapsto h(\nu):= \k_\nu$ is smooth on $S^{N-1}$. Moreover, the function $\partial_r u$ is radial and nonpositive, so we may write $\partial_r u(x)= -U(|x|)$ with a function $U:(0,\infty) \to [0,\infty)$. We may then apply Corollary \ref{appendix-corollary-sphere} in the appendix to obtain the representation
\begin{equation}
  \label{eq:prelim-representation}
w(x) = w_*(x) + h_* \partial_r u(x) \frac{\nu_*\cdot x}{|x|}= w_*(x) + h_* \nabla u(x) \cdot \nu_*  \qquad \text{for $x \in \R^N$}  
\end{equation}
with a fixed vector $\nu_* \in S^{N-1}$, a constant $h_* \in \R$ and a radial function $w_*: \R^N \to \R$. Since $w \in \H^s_{nr}(\R^N)$ it follows that $w_* \equiv 0$. Hence $w \in \spann \{\de_{x_1}u,\dots,\de_{x_N}u\}$, as desired.
\end{proof}
\begin{remark}\label{lem:rema-gen-non}
We note that the above argument to prove Theorem \ref{th:nondeg} is flexible enough to be applied to positive solutions $u\in H^s(\R^N)$ to the corresponding relativistic problem  $(-\D+m^2)^s u =u^p$ in $\R^N$ with $m>0$.
 This follows from the fact that Lemmas~\ref{lem:qual-sol} and \ref{lem:Picone} still hold if we replace $|x-y|^{-N-2s}$ with $K(|x-y|)$ for some strictly decreasing function $K$. Moreover, $K$ should be chosen such that $u\mapsto \int_{\R^N}\int_{\R^N}(u(x)-x(y))^2K(|x-y|)dxdy+\|u\|_{L^2(\R^N)}^2$ is comparable to $\|u\|_s^2$ on $ H^s(\R^N)$. We also recall from \cite{Fall-Felli-2} that,  for all $u\in H^s( \R^N)$, 
$$
\int_{\R^N}u(x)[(-\D+m^2)^s-m^{2s}]u(x) dx={ c_{N,s}'}\int_{\R^N}\! \int_{\R^N}\!{(u(x)-u(y)) ^2  }\cK_{\frac{N+2s}{2}}^m(m|x-y|)dx dy,
$$
 where $ c_{N,s}':= \frac{s 2^{\frac{2s+2-N}{2}}}{\pi^{\frac{N}{2}}\G(1-s)}$ and    $ \cK_\nu^m:(0,\infty)\to (0,\infty)$  is given by 
$$
 \cK_\nu^m( r)= m^{2\nu}  r^{-\nu }K_\nu(r),
$$
and  $K_\nu$ is the modified Bessel function of the second kind.  From the identity  $K_\nu'(r)=-\frac{\nu}{r}K_\nu -K_{\nu-1}$  and the fact that $K_{-\nu}=K_\nu$ for $\nu>0$, we see that  $ \cK_\nu^m$ is strictly decreasing on $(0,\infty)$.
\end{remark}

\section{Radial nondegeneracy in  dimension $N=1$}
\label{sec:radi-nond-spec}

In this section, we shall prove the following radial nondegeneracy result in the special case $N=1$.

\begin{theorem}
\label{radial-nondegeneracy-combined}  
Let $\lambda \ge 0$, and let $u \in \cH^s(B)$ be a  solution of 
\begin{equation}
    \label{eq:1.1-simplified-intervall}
\Ds u + \lambda u = u^{p}\quad\text{ in ${B}=(-1,1)$},\qquad u>0 \quad\textrm{in ${B}$}, \quad u = 0\quad\text{in $\R \setminus {B}$.}    
  \end{equation}
  Then we have:
  \begin{itemize}
  \item[(i)] $\Lambda_2 \not = p$ for the second eigenvalue $\Lambda_2$ of the weighted eigenvalue problem~\eqref{eq:weighted-eigen}.
  \item[(ii)] If $\lambda=0$, then $\Lambda = p$ is not an eigenvalue of the weighted eigenvalue problem~\eqref{eq:weighted-eigen}.
  \end{itemize}
\end{theorem}
In the special case $p=2$ and $N=1$, we have the following result in the case of problem \eqref{eq:1.1}.

\begin{theorem}
\label{radial-nondegeneracy-p2B}  
Let $\lambda >0$, $s\in (\frac{1}{6},1)$ and let $u \in \cH^s(B)$ be a   solution of 
\begin{equation}
    \label{eq:1.1-simplified-intervall-p2B}
\Ds u + \lambda u = u^{2}\quad\text{ in ${B}=(-1,1)$},\qquad u>0 \quad\textrm{in ${B}$}, \quad u = 0\quad\text{in $\R \setminus {B}$.}    
  \end{equation}
  Then $\Lambda = p$ is not an eigenvalue of the weighted eigenvalue problem~\eqref{eq:weighted-eigen}.
\end{theorem}

In the case of problem \eqref{eq:1.1RN}, we have 
\begin{theorem}
\label{radial-nondegeneracy-p2R}  
Let $\lambda > 0$,  $s\in (\frac{1}{6},1)$ and let $u \in H^s(\R)$ be an even  solution of 
\begin{equation}
    \label{eq:1.1-simplified-intervall-p2R}
\Ds u + \lambda u = u^{2}\quad\text{ in $\R$},\qquad u>0 \quad\textrm{in ${\R}$}.  
  \end{equation}
  Then the eigenspace of the eigenvalue $\Lambda = p$  for the weighted eigenvalue problem~\eqref{eq:weighted-eigen} is spanned by $u'$.
\end{theorem}
In order to prove the above theorems in a combined manner,   we will  consider $B_R:=(-R,R)$ for $R\in (0,+\infty]$ and  $u\in \cH^s(B_R)$  a positive even solution to  
\be \label{eq:1.1R-BR}
\Ds u+\l u=u^p \qquad\textrm{ in $B_R$}
\ee
 with $p\in( 1, 2^*_s-1)$ and $\l>-\l_1(B_R)$, where $\l_1(B_R)$ is the first Dirichlet eigenvalue of $(-\Delta)^s$ on $B_R$ if $R< \infty$ and $\l_1(\R)=0$. Since we are interested in radial nondegeneracy,  we will mainly focus on the study of even solutions $w\in \cH^s(B_R)$ to 
 the problem
  \begin{equation}
    \label{eq:lin-simple-lambda-0-1}
  (-\Delta)^s w + \lambda w = p u^{p-1}w \quad \text{in $B_R$.}     
\end{equation}
 
Our first observation is the following.

\begin{lemma}\label{lemma-radial-boundary-general-eigenfunction-M1}
Let $w \in \cH^s(B_R)$ be an  even  solution of (\ref{eq:lin-simple-lambda-0-1}). 
If $R=+\infty$, then 
\be\label{eq:integ-parts-psiuw-inty} 
\int_{\R} u w\,dx=0.
\ee
If $R<+\infty$, then 
\be\label{eq:integ-parts-psiuw} 
  \psi_w(R) = - \frac{s \lambda R }{\Gamma^2(1+s)\psi_u(R)}\int_{B_R} u w\,dx,
\ee
where $\psi_w$ and $\psi_u$ denote the fractional normal derivatives of $w$ and $u$ defined in Remark~\ref{rem-fractional-derivative}.
\end{lemma}

\begin{proof}
We start with the proof of \eqref{eq:integ-parts-psiuw}. 
  By (\ref{eq:1.1R-BR}) and the fractional integration by parts formula, see \cite[Theorem 1.9]{RX-Poh} or  \cite[Theorem 1.3]{Djitte-Fall-Weth}, we have 
\begin{align}
  &\int_{B_R}x u' (p u^{p-1} w-\lambda w)\,dx= \int_{B_R}x u'\Ds w\,dx \nonumber\\
  &=-\int_{B_R}x w' \Ds u\,dx-2R \G^2(1+s)\psi_u(R)\psi_w(R) -(1-2s)[u,w]_s.
\label{eq-integration-by-parts20-lambda}
\end{align}
Moreover, by (\ref{eq:1.1-simplified-intervall}) and \eqref{eq:weighted-eigen-simplified},
$$
[u,w]_s =\int_{B_R} (u^p-\lambda u) w\,dx  =  \int_{B_R} u(p u^{p-1}w-\lambda w)\,dx, 
$$
which implies that 
\begin{equation}
\label{orthogonality-u-p}
\int_{B_R} u^p w\,dx = 0 \qquad \text{and}\qquad  [u,w]_s =- \lambda \int_{B_R}  u w\,dx.
\end{equation}
In addition, by (\ref{eq:1.1-simplified-intervall}) and integration by parts, 
\begin{align*}
-\int_{B_R}x w' \Ds u\,dx&=  -\int_{B_R}x w' (u^p-\lambda u)\,dx=\int_{B_R} x u' (p u^{p-1}-\lambda)w\, dx+ \int_{B_R}w (u^p-\lambda u)\, dx\\
&=\int_{B_R} x u' (p u^{p-1}-\lambda)w\, dx-\lambda  \int_{B_R}u w \, dx.
\end{align*}
Combining these identities gives
$$
2R \G^2(1+s)\psi_u(R)\psi_w(R)= -2s \lambda \int_{B_R} u w\,dx.
$$
Now \eqref{eq:integ-parts-psiuw} follows since $\psi_u(R)>0$ by Remark~\ref{rem-fractional-derivative}.\\
Moreover, \eqref{eq:integ-parts-psiuw-inty} follows from the identity
\begin{align*}
  &\int_{\R}x u' (p u^{p-1} w-\lambda w)\,dx= \int_{\R}x u'\Ds w\,dx =-\int_{\R}x w' \Ds u\,dx  -(1-2s)[u,w]_s
\end{align*}
and the same integration by parts  as above.
\end{proof}
In the following, we need to consider the \textit{$s$-harmonic extension} $V$ of a function $v \in \cH^s(B_R)$, which has been introduced in \cite{CaSi} and is sometimes called the Caffarelli-Silvestre extension. 

For this we let $\R^{2}_+=\{(x,t)\in \R^2\,:\, x\in \R, t>0\}$, and we consider the function space
$D^{1,2}(\R^{2}_+;t^{1-2s})$ defined as the completion of $C^\infty(\ov{\R^{2}_+})$ with respect to the norm
$$
V \mapsto \Bigl(\int_{\R^{2}_+}|\n V|^2 t^{1-2s}\, dxdt\Bigr)^{\frac{1}{2}} .
$$
We also consider the space
\be 
D^{1,2}_{B_R}(\R^2_+;t^{1-2s})=\{V\in D^{1,2}(\R^2_+;t^{1-2s}) \,:\, V(\cdot,0)=0 \textrm{ on $\R \setminus B_R$}\}.
\ee
Then we recall that the \textit{$s$-harmonic extension} $V \in D^{1,2}_{B_R}(\R^2_+;t^{1-2s})$ of a function $v \in \cH^s(B_R)$ is given as the unique minimizer of the functional
$$
V \mapsto \int_{\R^{2}_+}|\n V|^2 t^{1-2s}\, dxdt  
$$
within the class of functions $V \in D^{1,2}_{B_R}(\R^2_+;t^{1-2s})$ with $V \big|_{\R \times \{0\}}= v$ in trace sense. Here we identify $\R$ with $\R \times \{0\}$. The function $V$ therefore satisfies
\begin{equation}
\label{eq:extens1-equation}
\div(t^{1-2s}\n V)= 0
\end{equation}
together with the boundary condition
\begin{equation}
\label{eq:extens1-bc}
\lim \limits_{t\to 0} V(x,t)=v (x) \quad\textrm{ for a.e. $x\in \R$.}
\end{equation}
Moreover by \cite{CaSi},  $V$ satisfies
\begin{align}
  \label{eq:weak-sense-extension-general}
  \int_{\R^2_+}t^{1-2s} \nabla V \cdot \nabla \Phi d(x,t) &= d_s [v,\varphi ]_s\\
  &\text{for all $\Phi \in D^{1,2}_{B_R}(\R^2_+;t^{1-2s})$ with $\varphi = \Phi\big|_{\R \times \{0\}} \in \cH^s(B_R),$}\nonumber    
\end{align}
where $[\cdot,\cdot]_s$ is defined in (\ref{eq:def-gagliardo-nirenberg-quadratic-form}) and the constant $d_s$ is given by $d_s = 2^{2s-1} \frac{\Gamma(s)}{\Gamma(1-s)}$, see e.g. \cite{CS}.

In addition, the function $V$ has the Poisson kernel representation 
\begin{equation}
  \label{eq:poisson-kernel}
 V(x,t)=p_{1,s} t^{2s}\int_{\R}\frac{v(y) dy}{(t^2+|x-y|^2)^{\frac{1+2s}{2}}}
\end{equation}
with
$$
p_{1,s} =\Bigl(\int_{\R}(|z|^2+1)^{-\frac{1+2s}{2}}\,dz\Bigr)^{-1} = \pi^{-1/s}s\frac{\Gamma(\frac{1}{2}+s)}{\Gamma(1+s)}.
$$
In the case where $v \in \cH^s(B_R) \cap C^\infty(B_R) \cap C^s(\R)$, we also have the regularity properties 
 \begin{equation}
   \label{eq:W-reg-prop}
 V \in C^s_{loc}(\overline{\R^2_+}) \cap C^\infty(\overline{\R^2_+} \setminus (\partial B_R \times \{0\})) .
 \end{equation}
Moreover, in this case we have the pointwise property 
\begin{equation}
  \label{eq:pointwise-normal-deriv-prop}
- \lim_{t\to  0}t^{1-2s}\de_t V (x,t)=d_s \Ds v(x) \qquad\textrm{ for all $x\in B_R$.}
\end{equation}
In particular, (\ref{eq:W-reg-prop}) and (\ref{eq:pointwise-normal-deriv-prop}) are true for the $s$-harmonic extension $W$ of the eigenfunction $w$. In the following, we need to analyze the nodal structure of $W$.

\begin{definition}
\label{def-nodal-domain-extension}  
We call a subset 
$\cO \subset \overline{\R^2_+}$ a nodal domain of $W$ if $\cO$ is a connected component of the set $\{(x,t) \in \overline{\R^2_+}\::\: W(x,t) \not = 0\}$.
\end{definition}
Since $W$ is continuous on $\overline{\R^2_+}$, every nodal domain $\cO$ is relatively open in $\overline{\R^2_+}$. We first note the following.

\begin{lemma}
  \label{lem-nod-dom-simple}
  Let $w \in \cH^s(B_R)$ be a fixed  nontrivial solution of (\ref{eq:lin-simple-lambda-0-1}). 
  Let $\cO \subset \ov{\R^2_+}$ be a nodal domain of the $s$-harmonic extension $W$ of $w$, and let $\tilde \cO:= \{x \in B_R \::\: (x,0) \in \cO\}$. Then we have 
\begin{equation}
  \label{eq:weak-sense-extension-nodal}
\int_{\cO}t^{1-2s} |\nabla W|^2d(x,t) = d_s \int_{\tilde \cO} (p  u^{p-1}-\lambda) w^2(x)\,dx .
\end{equation}
Moreover,  $\tilde \cO \not = \varnothing$.
\end{lemma}

\begin{proof}
Applying (\ref{eq:weak-sense-extension-general}) to $v= w$, $V=W$ and the function $\Phi = 1_{\cO}W \in D^{1,2}_{B_R}(\R^2_+;t^{1-2s})$ gives 
\begin{align*}
  &\int_{\cO}t^{1-2s} |\nabla W|^2d(x,t) = \int_{\R^2_+}t^{1-2s} \nabla W \cdot \nabla \Phi d(x,t) = d_s [v,\varphi ]_s\\
  &=d_s  \int_{B_R}(p  u^{p-1}-\lambda)w \varphi\,dx = d_s \int_{\tilde \cO}(p  u^{p-1}-\lambda) w^2(x)\,dx    
\end{align*}
with 
$$
\varphi= \Phi\big|_{\R \times \{0\}} = 1_{\tilde \cO}w \in \cH^s(B_R). 
$$
This yields \eqref{eq:weak-sense-extension-nodal}. Moreover, assuming by contradiction that $\tilde \cO=\varnothing$ yields
$$
\int_{\cO}t^{1-2s} |\nabla W|^2d(x,t)=0,
$$
which in turn implies that $W$ is constant in $\cO$. Hence, by continuity, $W\equiv 0$ in $\cO$ which is not possible. Hence $\tilde \cO \not =\varnothing$, as claimed.
\end{proof}

We also need the following important consequence of Theorem~\ref{th:nondeg}(i).

\begin{lemma}
  \label{consequence-odd-nondegeneracy}
  Let $V \in D^{1,2}_{B_R}(\R^2_+;t^{1-2s})$ be a  function which is odd with respect to the reflection at the hyperplane $T_0:= \{(0,t)\::\: t \in \R\}$ such that 
\be \label{eq:Equility-in-odd-extension}
  \int_{\R^2_+}t^{1-2s} |\nabla V|^2d(x,t) = d_s \int_{B_R} (p u^{p-1}-\l)v^2\,dx  .
\ee
 If $R<+\infty$ then $V\equiv 0$ and if $R=+\infty$ then $V(\cdot, 0)$  is proportional to $u'$.
\end{lemma}

\begin{proof}
Let $\ti V$ denote the $s$-harmonic extension of $v:=V(\cdot, 0)$. The variational characterization of the $s$-harmonic extension and \eqref{eq:Equility-in-odd-extension},  then give
\begin{equation}
  \label{consequence-odd-nondegeneracy-2}
 d_s \int_{B_R} (p u^{p-1}-\l)v^2\,dx =  \int_{\R^2_+} t^{1-2s} |\nabla V|^2d(x,t) \ge \int_{\R^2_+}t^{1-2s} |\nabla \ti V|^2d(x,t) = d_s [v]_s^2.
\end{equation}
If $v\not\equiv 0$, since it is odd, then  Theorem~\ref{th:nondeg}(i) implies
\begin{equation}
  \label{consequence-odd-nondegeneracy-3}
  [v|_s^2 > \int_{B_R} (p u^{p-1}-\l) v^2\,dx
\end{equation}
and thus if $R<+\infty$ we get a contradiction with  \eqref{consequence-odd-nondegeneracy-2}.  Hence $v\equiv 0$ and thus $V\equiv 0$. \\
 If now $R=+\infty$ then Theorem~\ref{th:nondeg}(ii) yields the inequality $$
[v|_s^2 \geq  \int_{B_R} (p u^{p-1}-\l) v^2\,dx
$$
and thus  equality holds in \eqref{consequence-odd-nondegeneracy-2} so that $v$ is proportional to $u'$.
\end{proof}

We now continue our analysis of the nodal structure of the $s$-harmonic extension $W$ of the eigenfunction $w$.
\begin{lemma}  
  \label{criterion-nodal-domain}
  Let $w \in \cH^s(B_R)$ be an even  nontrivial solution of (\ref{eq:lin-simple-lambda-0-1}) and let  $W$ be its $s$-harmonic extension.
  \begin{itemize}
  \item[(i)] No nodal domain of $W$ is contained in the set $H_{+}:= \{(t,x) \in \R^2\::\: x > 0, t \ge 0\}$.   
  \item[(ii)] If $x \in (0,R)$ satisfies $w(x) \not = 0$, then the points $(\pm x,0)$ belong to the same nodal domain $\cO$ of $W$.
  \item[(iii)] For every nodal domain $\cO$ of $W$ there exist a point $x \in (0,R)$ with $(\pm x,0) \in \cO$.
  \end{itemize}
\end{lemma}

\begin{proof}
(i)  Suppose by contradiction that $W$ has a nodal domain $\cO_+ \subset H_{+}$. Since $w$ is an even function, the function $W$ is even with respect to reflection $\sigma$ at the line $T_0:= \{(0,t)\::\: t \in \R\}$. Consequently, $\cO_-:=  \sigma(\cO_+)$ is also a nodal domain of $W$ with $\cO_- \not = \cO_+$. Hence we may define the function $V:= 1_{\cO_+} W- 1_{\cO_-} W \in  D^{1,2}_{B_R}(\R^2_+;t^{1-2s})$ which is odd with respect to reflection $\sigma$ at the hyperplane the hyperplane $T_0$,
  and we set $v:= V(\cdot,0) \in \cH^s(B_R)$. With $\tilde \cO_{\pm}:= \{x \in B_R\::\: (x,0) \in\ov{ \cO_\pm}\}$, it then follows from Lemma~\ref{lem-nod-dom-simple} that $\tilde \cO_{\pm}\not=\emptyset$ and 
  \begin{align*}
& d_s  \int_{B_R} (p u^{p-1}-\l)v^2\,dx =   d_s  \int_{\tilde \cO_+} (p u^{p-1}-\l) w^2\,dx + d_s  \int_{\tilde \cO_-}(p u^{p-1}-\l) w^2\,dx\\
    &= \int_{\cO_+}t^{1-2s} |\nabla W|^2d(x,t) + \int_{\cO_-}t^{1-2s} |\nabla W|^2d(x,t)= \int_{\R^2_+}t^{1-2s} |\nabla V|^2d(x,t).
  \end{align*}
  Now by Lemma~\ref{consequence-odd-nondegeneracy}, $v\equiv 0$ if $R<+\infty$ and $v$ is proportional to $u'$ if $R=+\infty$. In both cases, we get a contradiction since $u'<0$ on $(0,+\infty)$.
  The claim thus follows.\\
  (ii) By continuity, there exists a unique nodal domain $\cO$ of $W$ in $\R^2_+$ with the property that $(x,0) \in \cO$.  By (i), we know that  $\cO \not \subset H_{+}$. Therefore $\cO$ must contain points in the half line $\{0\} \times \R_+$. Thus, by the evenness of the function $W$ with respect to the reflection $\sigma$, it follows that also $(-x,0) \in \cO$, as claimed.\\
  (iii) We already know from Lemma~\ref{lem-nod-dom-simple} that $\cO$ contains a point $x \in B_R$. If $x \in (0,R)$, the claim follows directly from (ii). If $x \in (-R,0)$, we consider the reflection $\cO'= \sigma(\cO)$ of $\cO$, so that $-x \in (0,R)$ and $(-x,0) \in \cO'$. Then the claim holds for $\cO'$ in place of $\cO$ by (ii). It thus follows that $(\pm x,0) \in \cO' \cap \cO$, and therefore $\cO$ and $\cO'$ coincide. Finally, if $x = 0$, then also $w(x+ \eps) \not = 0$ and $x+ \eps \in \cO$ for $\eps>0$ sufficiently small, so we may replace $x$ by $x+ \eps$ and obtain the claim.  
\end{proof}

\begin{lemma}
  \label{nodal-domains-extension-corollary}
   Let $w \in \cH^s(B_R)$ be an even   nontrivial solution of (\ref{eq:lin-simple-lambda-0-1}). 
If $0<x<y<R$ are points with $w(x) \not= 0$, $w(y) \not = 0$, and if the function $w$ has a zero in $(x,y)$, then $(x,0)$ and $(y,0)$ are contained in different nodal domains $\cO_x, \cO_y$ of the $s$-harmonic extension $W$ of $w$.
\end{lemma}

\begin{proof}
  Let $\cO_x$ resp. $\cO_y$ denote the nodal domains of $W$ which contain the points $(x,0)$ and $(y,0)$, respectively. If $w(x)$ and $w(y)$ have different sign, then $W$ has a different sign on $\cO_x$ and $\cO_y$ and therefore $\cO_x$ and $\cO_y$ are different, as claimed. Suppose now that $w(x)$ and $w(y)$ have the same sign, say, $w(x)>0$ and $w(y)>0$. Moreover, we suppose by contradiction that $\cO_x = \cO_y$. Then there exists a continuous curve $\gamma:[0,1] \to \cO_x$ joining the points $x$ and $y$. By assumption, there exists a point $r \in (x,y)$ with $w(r)=0$. Since $\gamma([0,1])$ is compact and does not contain the point $(r,0)$, we have 
    $$
    d:= \dist((r,0),\gamma([0,1]) >0.
    $$
    We claim that $(r,0)$ is not a local minimum of $W$ in $\overline{\R^2_+}$. Suppose by contradiction that there exists a relative neighborhood $N$ of $(r,0)$ in $\overline{\R^2_+}$ with $W(z) \ge W(r,0)=w(r)=0$ for $z \in N$. By the strong maximum principle for equation~(\ref{eq:extens1-equation}) with $W$ in place of $V$, we then have $W>0$ in $N \cap \R^2_+$ or $W \equiv 0$ in $N \cap \R^2_+$. The latter implies that $W\equiv 0$ in $\ov{\R^2_+}$ by unique continuation (see \cite{Fall-Felli}), which is not possible since $w \not \equiv 0$. Hence $W>0$ in $N \cap \R^2_+$ and therefore (\ref{eq:pointwise-normal-deriv-prop}) and  the Hopf type boundary lemma for equation (\ref{eq:extens1-equation}) given in \cite[Proposition 4.11]{CS} implies that 
$$
0> - \lim_{t\to  0}t^{1-2s}\de_t W (r,t)=d_s \Ds w(r) = (pu^{p-1}(r)-\lambda)w(r)= 0,
$$
which also yields a contradiction. Hence $(r,0)$ is not a local minimum of $W$, and therefore there exists a point $z  \in \overline{\R^2_+}$ with $W(z)<0$ and $|z-(r,0)|< d$. Then the line segment $L$ joining the points $z$ and $(r,0)$ does not intersect the curve $\gamma$. We denote by $\cO_-$ the nodal domain of $W$ containing the point $z$. Then by Lemma~\ref{criterion-nodal-domain}(iii), there exists a point $\rho < 0$ with $(\rho,0) \in \cO_-$. Moreover, there exists a continuous curve $\tilde \eta:[0,1] \to \cO_-$ joining the points $(\rho,0)$ and $z$. By adding the line segment $L$, we arrive at a continuous curve $\eta: [0,1] \to \ov{\R^2_+}$ joining the points $(\rho,0)$ and $(r,0)$. Since $\rho < x < r < y$, we may apply Lemma~\ref{sec:topological-lemma-1} from the appendix to see that the curves $\gamma$ and $\eta$ must intersect. Since, as noted above, $\gamma$ does not intersect the line segment $L$, the curves $\gamma$ and $\tilde \eta$ must have an intersection point. This however is impossible, as $\cO_x \cap \cO_- = \varnothing$ since $W$ is positive in $\cO_x$ and negative in $\cO_-$. The contradiction shows that $\cO_x \not = \cO_y$, as claimed. 
\end{proof}

\begin{lemma}
  \label{consequence-odd-nondegeneracyR}

Let  $w \in \cH^s(B_R)$ be  an  even nontrivial  solution to   \eqref{eq:lin-simple-lambda-0-1}.
Then $\int_{\R}w dx\not=0$.  If $R<+\infty$ and $\int_{\R}w dx>0$, then 
$$
\psi_w(R)>0\qquad\textrm{ and } \qquad \Ds w<0\quad\textrm{ on $\R\setminus \ov{B_R}$.}
$$
\end{lemma}
\begin{proof}
Let  $W$ be the $s$-harmonic extension of $w$. 
By the variant of the classical Courant nodal domain theorem given in \cite[Prop. 5.2]{FLS}, $W$ has a finite number of nodal domains and by Lemma \ref{nodal-domains-extension-corollary},  we have that  $w=W(\cdot,0)$ changes sign a finite number of times on $\R$.  Up to replacing  $W$ with $-W$ if necessary,   we can assume that there exists  $ x_*\in (0,R)$ such that
$$
w(x_*) = W(x_*,0)>0 \qquad \text{and}\qquad w(x)= W(x,0)\geq0 \quad \text{for all $x> x_*$.}
$$
We let $\O_+$ be the nodal domain of $W$  containing $ (x_*,0)$.  By Lemma \ref{consequence-odd-nondegeneracyR}, we have that $(-x_*,0)\in \O_+$.  Hence  there  exists a {simple}  continuous  curve $ \g:[0,1]\to \O_+ $ with $\g((0,1))\subset \R\times (0,+\infty)$,  $ \g(0)=(x_*,0)$ and $ \g(1)=(-x_*,0)$.
Then $\gamma$ separates the set $\R\times (0,+\infty)$ into a bounded domain and an unbounded  domain\footnote{Eventually, we can   reflect the curve $ \g$ with respect to the line $\R\times \{0\}$ to get a closed simple curve $\ov \g$ and then use the Jordan curve theorem. }.
   We  call $\cO$ the unbounded one, and we note  that 
$$
\de \cO=\cL:=(-\infty,-x_*)\cup \g([0,1])\cup(x_*,+\infty).
$$
 We let $f \in C^\infty_c(-x_*,x_*) \setminus \{0\}$   with $f \ge 0$   on $(-x_*,x_*)$, and we let $\phi \in \cH^s(B_R)$ be the unique weak solution of
  $$
  (-\Delta)^s \phi+\l \phi  = f \quad \text{in $B_R$.}
  $$
  By the fractional strong maximum principle (see e.g. \cite{JW-2019} or \cite[Prop. 3.3 and Rem. 3.5]{FJ-2015}), we then have $\phi>0$ in $B_R$. Moreover, letting $\Phi\in D^{1,2}_{B_R}(\R^2_+;t^{1-2s})$ denote the $s$-harmonic extension of $\phi$, we have $\Phi>0$ in $\R^2_+$.  By the continuity of $W$ and $\Phi$,  the compactness of $\g([0,1])$ and the fact that $W$ is positive on $\gamma([0,1])$,  there exists $\mu>0$ such that 
$$
 W\geq \mu \Phi \quad\textrm{ on $\g([0,1])$}.
$$
Consequently, the function $H:=W-\mu \Phi\in D^{1,2}_{B_R}(\R^2_+;t^{1-2s})$ satisfies 
\begin{align*}
\begin{cases}
\div(t^{1-2s}\n H)= 0&\quad\textrm{ in $\cO$}\\
H\geq 0& \quad\textrm{ on $\g([0,1])  $,}\\
-t^{1-2s}\de_t H(x,0)+\l H (x,0)\geq  0&  \quad\textrm{ for $x\in   B_R\setminus [-x_*,x_*]$.}
\end{cases}
\end{align*}
where, we used that   $f\equiv 0$ on $  B_R\setminus [-x_*,x_*]$.     We can thus apply the weak maximum principle to deduce that  
\be \label{eq:ineq-okWPhi}
W\geq \mu \Phi \qquad\textrm{ in $\ov{\cO}$.}
\ee
In particular,  since $\{(0,t),\:\, t>t_*\}\subset \cO$, for some large $t_*>0$, we have 
\begin{equation}
  \label{eq:w-Phi-comparison-asymptotic}
t^{2s}p_{1,s}   \int_{B_R}\frac{w(y) dy}{(t^2+|y|^2)^{\frac{1+2s}{2}}}\geq  \mu p_{1,s} t^{2s}   \int_{B_R}\frac{\phi(y) dy}{(t^2+|y|^2)^{\frac{1+2s}{2}}} \qquad \text{for all $t>t_*$.}  
\end{equation}
  
By \cite[Lemma C.2 (i)]{FLS}, we have that $w,\phi \in L^1(\R)$ if $R=+\infty$ and obviously $w,\phi\in L^1(B_R)$ if $R<+\infty$.   Hence, by (\ref{eq:w-Phi-comparison-asymptotic}) and the dominated convergence theorem, it follows that  
$$
\int_{B_R} w dy = \lim_{t \to +\infty}   \int_{B_R}\frac{w(y) dy}{(1+(|y|/t)^2)^{\frac{1+2s}{2}}}\ge \lim_{t \to +\infty}\int_{B_R}\frac{\mu\phi(y) dy}{(1+(|y|/t)^2)^{\frac{1+2s}{2}}} =\mu  \int_{B_R}\phi dy  >0,
$$
which ends the first part of the  proof. \\
If now $R<+\infty$, then we deduce from \eqref{eq:ineq-okWPhi} that 
\be \label{eq:ineq-okWPhi-trace}
w(x)=W(x,0) \geq  \mu \Phi(x,0)=\mu  \phi(x) \quad\textrm{ for all $ x\in  B_R\setminus [-x_*,x_*]$.}
\ee
In addition,  by the fractional Hopf Lemma (see e.g. \cite[Proposition 3.3]{FJ-2015}), we have that $\psi_\phi(R)=\lim \limits_{|x|\nearrow R}\frac{\phi(x)}{(R-|x|)^s}>0$.  From this and \eqref{eq:ineq-okWPhi-trace}, we conclude that $\psi_w(R)>0$, as desired.   Furthermore, by the Hopf type boundary lemma given in \cite[Proposition 4.11]{CS}, we have $t^{1-2s}\de_t \Phi(x,0)>0$ for all $x\in \R\setminus \ov{B_R}$.  Using this and  \eqref{eq:ineq-okWPhi} we obtain
$$
d_s \Ds w=-t^{1-2s}\de_t W(\cdot ,0)\leq-\mu t^{1-2s}\de_t \Phi(\cdot ,0) <0 \qquad \text{on $ \R\setminus \ov{B_R}$.}
$$
\end{proof}

We are now in a position to complete the proof  of the nondegeneracy results.

\begin{proof}[Proof of Theorem~\ref{radial-nondegeneracy-combined}  (completed)]
We first assume that $\lambda>0$, and that $\Lambda_2 \ge p$ for the second eigenvalue of the weighted eigenvalue problem \eqref{eq:weighted-eigen}. 
Assuming, as before, by contradiction the existence of a fixed even nontrival solution $w \in \cH^s(B_R)$ of (\ref{eq:lin-simple-lambda-0-1}) with $R=1$ and recalling that $\Lambda_1 = 1$ with eigenfunction $u$, we then deduce that $\Lambda_2 = p$. Thus a variant of the classical Courant nodal domain theorem given in \cite[Prop. 5.2]{FLS} shows that $W$ has precisely two nodal domains $\cO_\pm$ with $W>0$ in $\cO_+$ and $W<0$ in $\cO_-$. It then follows from Lemmas~\ref{criterion-nodal-domain} and \ref{nodal-domains-extension-corollary} that $w$ changes sign exactly once in $(0,1)$. So, after replacing $w$ with $-w$ if necessary, there exists a point $r_0 \in (0,1)$ with
\begin{equation}
  \label{eq:r-0-property}
\text{$w \ge 0$, $w \not \equiv 0$ in $[-r_0,r_0]\qquad $ and $\qquad w \le 0$, $w \not \equiv 0$ in $[-1,-r_0] \cup [r_0,1]$.}  
\end{equation}
Combining this information with Lemma~\ref{lemma-radial-boundary-general-eigenfunction-M1}, we find that
  $$
  0 \ge   \psi_w(1) = - \frac{s \lambda}{\Gamma^2(1+s)\psi_u(1)}\int_{B} u w\,dx
  $$
  and therefore,  provided that  $\l>0$, 
  \begin{equation}
    \label{eq:u-w-positive-integral}
  \int_B u w\,dx \geq 0.
  \end{equation}
  On the other hand, since the function $u^{p-1}$ is positive, even and strictly decreasing in $|x|$, there exists $\kappa>0$ with the property that
  $u^{p} - \kappa u >0$ for $|x| <r_0$ and $u^{p} - \kappa u <0$ for $r_0 < |x|<1$, which by (\ref{orthogonality-u-p}) and (\ref{eq:r-0-property}) implies that 
  $$
  0 < \int_{B}(u^{p} - \kappa u) w \,dx = - \kappa  \int_{B} u w\,dx.
  $$
  This contradicts (\ref{eq:u-w-positive-integral}), and the contradiction finishes the proof of the theorem in the case $\l>0$ and $\L_2\geq p$, which is (i).\\[0.2cm]
  Finally, we consider the case where $\lambda=0$. In this case, Lemma~\ref{consequence-odd-nondegeneracyR} shows that $\psi_w(1)<0$, and this contradicts Lemma~\ref{lemma-radial-boundary-general-eigenfunction-M1}.
The proof of (ii) is thus finished.
\end{proof}
\begin{proof}[ Theorem \ref{radial-nondegeneracy-p2R} (completed)]
  We let $v \in \cH^s(\R)$ be a solution to  (\ref{eq:lin-simple-lambda-0-1}), and we define   $x\mapsto w(x):=\frac{v(x)+v(-x)}{2}$.
  Since $u$ is even,  we have that 
   \be
   \label{eq:lin-simple-lambda-0-R-ev}
    (-\Delta)^s w + \l  w = 2 u w \quad\textrm{  in $\R$. }
    \ee
Then   by Lemma~\ref{lemma-radial-boundary-general-eigenfunction-M1} we have 
\be\label{eq:perpL2}
\int_{\R} w(y)u(y)dy=0. 
\ee
Now integrating  \eqref{eq:lin-simple-lambda-0-R-ev} on $\R$,  we obtain
\begin{equation}
  \label{eq:integrate-equation-over-R}
0=\int_{\R}\Ds w(y)dy=-\l \int_{\R} w(y)dy+2\int_{\R} w(y)u(y)dy.
\end{equation}
Note here that the first equality in (\ref{eq:integrate-equation-over-R}) follows by approximation. Indeed, letting $\chi_n(x)=\chi( x/n)$ for $n \in \N$, where $\chi $ is given by \eqref{eq:def-chi}, we can multiply  \eqref{eq:lin-simple-lambda-0-1} with $\chi_n$ and integrate on $\R$.  Then using that $ \int_{\R}\Ds w(y)\chi_n (y)dy=  \int_{\R}w(y) \Ds  \chi_n (y)dy$ and the fact that  $\|\Ds  \chi_n \|_{L^\infty(\R)}\le C n^{-2s}$ with some constant $C>0$, we get $\int_{\R}\Ds w(y) dy=0$ from the dominated convergence theorem.

Combining (\ref{eq:integrate-equation-over-R}) with \eqref{eq:perpL2},  we get $\int_{\R} w(y)dy=0$.  It then follows   from Lemma \ref{consequence-odd-nondegeneracyR} that $w\equiv 0$ on $\R$.  We thus conclude that $v$ is an odd function.  Finally Theorem \ref{th:nondeg}$(ii)$ implies  that $v$ is proportional to $u'$.\\
\end{proof}
\begin{proof}[Proof of Theorem \ref{radial-nondegeneracy-p2B}    (completed)]
  We assume for simplicity that $R=1$ and we recall that $B=B_1=(-1,1)$.
  We let $v \in \cH^s(B)$ be a solution to (\ref{eq:lin-simple-lambda-0-1}) with $R=1$, and we consider again $x\mapsto w(x):=\frac{v(x)+v(-x)}{2}$, which satisfies
   \be
   \label{eq:lin-simple-lambda-0-R-ev-B}
    (-\Delta)^s w + \l  w = 2 u w \quad\textrm{  in $B$. }
    \ee
We first claim that 
\be\label{eq:claimB_R}
\int_{B}[ - \l w+2 uw](y)dy+ \int_{\R\setminus \ov B}\Ds w(y)dy=0. 
\ee
Let us postpone the proof of this claim and finish the proof of the Theorem. 
Combining Lemma \ref{lemma-radial-boundary-general-eigenfunction-M1} with \eqref{eq:lin-simple-lambda-0-R-ev-B}, we find that
\begin{align*}
%
0&=-\l \int_{\R} w(y)dy-\frac{2}{s \l} \psi_u(1)\psi_w(1)\G^2(1+s)+\int_{\R\setminus\ov  B}\Ds w(y)dy.
\end{align*}
Since $\l> 0$ and $\psi_u(1)>0$, Lemma \ref{consequence-odd-nondegeneracyR} now implies that $w\equiv 0$ on $\R$, so that $v$ is an odd function. Finally, Theorem \ref{th:nondeg}$(i)$ yields $v\equiv 0$  on $\R$,  as required.

It thus remains to prove \eqref{eq:claimB_R},  and for this we argue by approximation. Let $\z_k$ be given by \eqref{eq:def-z-k-etc}, and note that $\z_k w\in C^\infty_c(B) \subset C^\infty_c(\R)$. By  \eqref{eq:lin-simple-lambda-0-1}, we have 
\begin{align*}
&\Ds (\z_k w)=\z_k \Ds w+w \Ds \z_k-I(w,\z_k)\\
&=[ - \l w+2 uw]\z_k +w \Ds \z_k-I(w,\z_k) \qquad\textrm{ in $B$,}
\end{align*}
where $I(w,\z_k)(y) =c_{1,s}\int_{\R}\frac{(w(x)-w(y))(\z_k(x)-\z_k(y))}{|x-y|^{1+2s}}dx$.
Hence,  as in the proof of  Theorem \ref{radial-nondegeneracy-p2R} above,  we have 
\begin{align}\label{eq:I-before-lim}
&0=\int_{\R}\Ds (\z_kw)(y)dy=\int_{B}\Ds (\z_kw)(y)dy+ \int_{\R\setminus \ov B}\Ds (\z_k w)(y)dy \\
&=\int_{B}[ -\l \z_k w +2  \z_k wu] (y)dy+\int_{B}[w \Ds \z_k-I(w,\z_k)](y) dy  + \int_{\R\setminus \ov B}\Ds (\z_k w)(y)dy.\nonumber
\end{align}
By the dominated convergence theorem, we have
\be\label{eq:I11}
 \lim_{k\to +\infty}\int_{B}[ -\l \z_k w + \z_k w u] (y)dy=\int_{B}[ -\l  w +2  w u] (y)dy.
\ee
Using  that
$$
|\z_k(x)w(x)|\leq C (1-|x|)^s_+ \le C (|x-y|^s) \qquad \text{for all $x \in B$, $y\in \R\setminus \ov{B}$,}
$$
we have, for $y \in \R^N \setminus \overline{B}$, the estimate 
\begin{align}
|\Ds (\z_k w)(y)|&= c_{1,s}\left|\int_{\R} \frac{\z_k(x) w(x)}{|x-y|^{1+2s}} dx\right|\leq  C \int_B \frac{(1-|x|)^s}{|x-y|^{1+2s}} dx \nonumber\\
                 &\leq C \int_B \min \Bigl\{\frac{1}{|x-y|^{1+s}}, \frac{1}{|x-y|^{1+2s}} \Bigr\} dx \nonumber\\
  &\le C \Bigl( (1-|y|)^{-s} 1_{B_2 \setminus\ov  B }(y)+ (1+|y|)^{-1-2s} 1_{\R\setminus B_2}(y)\Bigr). \label{integrable-RHS}
\end{align}
Here, $C$ denotes a positive constant which may change from line to line in the following.
Furthermore,  it is straightforward to check that
$$
\lim_{k \to  +\infty} \Ds (\z_k w)(y)=  \Ds w(y) \qquad \text{for all $y\in \R\setminus \ov{B}$.}
$$
Since the RHS of (\ref{integrable-RHS}) is in $L^1(\R)$, the dominated convergence theorem  yields that
\be\label{eq:I12}
 \lim_{k\to+\infty}  \int_{\R\setminus \ov B}\Ds (\z_k w)(y)dy=\int_{\R\setminus \ov B} \Ds w(y) dy.
\ee
Next,  for $\e\in (0,1)$,  we write
\begin{align}
\int_{B}[w \Ds \z_k-I(w,\z_k)](y) dy&= \int_{B_{1-\e}}[w \Ds \z_k-I(w,\z_k) ](y)dy \nonumber\\
&+ \int_{B\setminus B_{1-\e}}[w \Ds \z_k-I(w,\z_k) ](y)dy. \label{eq:approx-bdr}
\end{align}
Following \cite[Section 6, p. 231]{DFW},  we have that 
\be\label{eq:lim-int-zero}
\lim_{k\to +\infty} \int_{B_{1-\eps}}[w \Ds \z_k-I(w,\z_k) ](y)dy=0 .
\ee
Moreover, by \cite[Lemma 6.3 and Lemma 6.8]{DFW}), there exists $\ov \eps \in (0,1)$ and $C>0$ with 
$$
|\Ds \z_k-I(w,\z_k)|(y) \leq C \left( \frac{k^{2s}}{1+ |k(1-|y|)|^{1+2s}}+ \frac{k^s}{1+|k(1-|y|)|^{1+s}} \right)
$$
for $1-\ov \eps \le |y| \le 1$.  Using this estimate together with the fact that $|w(y)|\leq  C(1-|y|)^s_+$ for $y \in B$, we have, by the change of variable $r = k(1-|y|)$,  
\begin{align*}
  &\int_{B \setminus B_{1-\ov \eps}}|w \Ds \z_k-I(w,\z_k) |(y)dy\\
  &\le
                                                                 C \int_{B \setminus B_{1-\ov \eps}}\left( \frac{k^{2s}(1-|y|)^s }{1+ |k(1-|y|)|^{1+2s}}+ \frac{k^s}{1+|k(1-|y|)|^{1+s}} \right)dy \\
                                                               &=\frac{C}{k} \int_{0}^{k\ov \eps} \left(\frac{k^{2s}}{1+ r^{1+2s}}\Bigl(\frac{r}{k}\Bigr)^s+ \frac{k^s}{1+r^{1+s}} \right)dr \le C k^{s-1} \int_{0}^{\infty} \frac{1}{1+r^{1+s}}dr \to 0 \quad \text{as $k \to +\infty$.}
\end{align*}
Combining this with  \eqref{eq:lim-int-zero}, we obtain
$$
\lim_{k\to +\infty} \int_{B}[w \Ds \z_k-I(w,\z_k) ](y)dy=0 .
$$
This implies, together with \eqref{eq:I-before-lim}, \eqref{eq:I11} and \eqref{eq:I12}, that \eqref{eq:claimB_R} holds, as required.
\end{proof}

\section{Uniqueness in the case of $B=(-1,1)$}
\label{sec:uniqueness-case-ball}

In this section, we shall finish the proofs of Theorems~\ref{new-uniqueness-intervall} and  of  Theorems~\ref{new-uniqueness-intervall-GS} by a continuation argument. We start with the following locally uniform estimates.
\begin{lemma}\label{lem:unif-bnd}
  Let $1<p_0<2^*_s-1$, and let $\lambda_0>-\l_1(B)$. Then there exists $\d>0$ and $C >0$ such that for all $p\in (p_0-\d, p_0+\d)$, $\l\in (\l_0-\d,\l_0+\d)$ and any  $u \in\cH^s{(B)}\cap C^s(\R^N)$ solving the problem
$$
(\cP_{p,\l}) \qquad\qquad \Ds u +\l u = u^{p}, \quad u>0 \qquad \text{ in ${B}$},\qquad u = 0 \qquad \text{on $\R \setminus B$.}
$$
 we have
\begin{itemize}
\item[(i)] $\|u\|_{L^\infty{(B)}}\leq C$;
\item[(ii)] $[u]_s+\|u\|_{C^s(\R^N)}\leq C$.
\end{itemize}
Moreover,  let $(p_n)_{n\in \N}$ and $(\l_n)_{n\in \N}$ be  sequences converging  to some $\ov p\in (1,2^*_s-1)$ and $\ov \l>-\l_1(B)$, respectively.  For $n\in \N$, we let  $u_n\in \cH^s(B)$ be a solution to  $(\cP_{p_n,\l_n})$.
Then,   $(u_n)_{n\in \N}$  possesses  a subsequence that  weakly (resp.  strongly) converges   in $\cH^s(B)$ (resp. in $C(\ov B)$)  to a solution $v$ of $ (\cP_{\ov p,\ov \l})$.  In particular if $u_n$ is a ground state solution of $(\cP_{p_n,\l_n})$ for all $n\in \N$, then $v$ is a ground state solution of $(\cP_{\ov p,\ov \l})$. 
\end{lemma}

\begin{proof}
Arguing by contradiction, we first suppose that there exist  sequences $p_n\to p_0$ and $\l_n\to\l_0$ and a sequence of positive  solutions $u_n\in \cH^s{(B)}\cap C^s(\R^N)$ of $(\cP_{p_n})$ such that  $b_n:=\|u_n\|_{L^\infty{(B)}}\to\infty$ as $n\to\infty$.   
We define 
$$
v_n(x):=\frac{1}{b_n}u_n(x/b_n^{\frac{p_n-1}{2s}}). 
$$
By Lemma \ref{lem:qual-sol}$(i)$,  $\|v_n\|_{L^\infty(\R^N)}=v_n(0)=1$.  By direct computations,
$$
\Ds v_n+\l_n b_n^{1-p_n} v_n=v_n^{p_n}\qquad \text{in $B_{r_n}(0)\quad $ with $\quad r_n:=b_n^{\frac{p_n-1}{2s}}$.}
$$
By fractional elliptic regularity theory (see \cite[Theorem 1.1]{RS16a}), there exists $\alpha>0$ such that the functions $v_n$ are uniformly bounded in $C^{2s+\a}(K)$ for any compact set $K\subset \R^N$. After passing to a subsequence, we may thus assume that $v_n \to \ov v$ in $C^{2s+\beta}_{loc}(\R^N)$ for $0 < \beta < \alpha$, where $\ov v$ satisfies $\Ds \ov v =\ov v^{p_0}$ in $\R^N$ and $\ov v(0)=1$.  Hence by \cite[Remark 1.9]{Jin-Li-Xiong}  we have that $v\equiv 0$ which is not possible. We thus conclude that there exists $\d>0$ and $C>0$ such that $\|u\|_{L^\infty(B)}\leq C$ for all $p\in (p_0-\d, p_0+\d)$, $\l\in (\l_0-\d,\l_0+\d)$ and all $u\in\cH^s{(B)}\cap C^s(\R^N)$ solving $(\cP_{p,\l})$. This proves (i).

Now (ii) follows from boundary regularity estimate in \cite[Theorem 1.2]{RS16a} and the uniform $L^\infty$ bound in (i), again after making $C$ larger if necessary.\\
Next, we note that, for  $\l>-\l_1(B)$,  by the Poincar\'e inequality for all  $u\in\cH^s{(B)} $ solving $(\cP_{p,\l})$, we have 
$$
\k(\l) [u]_{s}^2\leq [u]_{s}^2+\l\int_B u^2 dx \leq \int_{B} u^{p+1}dx,
$$
where $\k(\l):=\min \Bigl \{1, \Bigl(1+\frac{\l}{\l_1(B)}\Bigr)\Bigr\}>0$.
From this,   the Sobolev and H\"older inequalities,  for all $q\in ( p+1, 2^*_s)$ there exists $\ov C_q:=\ov C(N,s,q)>0$  such that 
$$
\k(\l)   \ov C_q\left( \int_{B} u^{q}dx\right)^{\frac{2}{q}}\leq [u]_{s}^2 +\l\int_B u^2 dx \leq  \int_{B} u^{p+1}dx \leq |B|^{1-\frac{p+1}{q}} \left( \int_{B} u^{q}dx\right)^{\frac{p+1}{q}}.
$$
We thus conclude that for every   $u\in\cH^s{(B)} $ solving $(\cP_{p,\l})$, we have
\be\label{eq:sobestim}
 \left( \int_{B} u^{q}dx\right)^{ \frac{p-1 }{q}}\geq \k(\l) \ov C_q |B|^{-1+\frac{p+1}{q}} .
\ee
Let $(p_n)_{n\in \N}$ and $(\l_n)_{n\in \N}$ be   sequences converging to  $\ov p\in (1,2^*_s-1)$ and $\ov\l>-\l_1(B)$, respectively,    and let $u_n\in \cH^s(B)$ be a solution of  $(\cP_{p_n,\l_n})$ for all $n\in \N$.  By (ii)   we have that,  up to a subsequence,    the sequence $(u_{n})_{n\in \N}$ converges weakly  in $\cH^s(B)$ and strongly in $ C(\ov B)$ to some  $v\geq 0$. Moreover $v\in \cH^s(B)\cap C(\ov B)$ weakly   solves $\Ds v+\ov \l v =v^{\ov p}$ in $B$.  We fix $q\in (\ov p+1,2^*_s-1)$. 
Then for large $n\in \N$ such that $q>p_n+1$,  \eqref{eq:sobestim} implies that,  
$$
 \left( \int_{B} u^{q}_n dx\right)^{ \frac{p_n-1 }{q}}\geq  \k(\l_n)\ov C_q |B|^{-1+\frac{p_n+1}{q}}.
 $$
Hence letting $n\to \infty$,  we see that $ \|v \|_{L ^{q}(B)} >0 $, so that $v \gneqq0$ in $B$. 
We then write   $\Ds v+V_v\, v= 0$ with the (frozen) potential $V_v =\ov  \l-v^{\ov p-1}\in L^\infty(B)$. Thus, by the fractional strong maximum principle (see e.g. \cite{JW-2019} or\cite[Prop. 3.3. and Rem. 3.5]{FJ-2015}), we have $v>0$ in $B$, and the lemma follows.
\end{proof}

We may now complete the proofs of Theorem~\ref{new-uniqueness-intervall} and Theorem~\ref{new-uniqueness-intervall-GS}.

\begin{proof}[Proof of  Theorem~\ref{new-uniqueness-intervall} (completed)]
From\footnote{This fact was proved in \cite[Theorem 1.6]{DIS}  for $N \ge 2$ however the same proof works also for $N=1$}  \cite[Theorem 1.6]{DIS}  there exists $p_0>1$ such that  for all $p\in (1,p_0)$ there exists a unique positive function  $u_p\in\cH^s{(B)}$ satisfying $\Ds u_p = u^p_p$ in ${B}=(-1,1)$ and $u_p$ is nondegenerate\footnote{ in the sense that the equation $\Ds w=p u_p^{p-1}w $ in $B$ has only the trivial solution in $\cH^s(B)$}.
 We let $p_*>1$ be the largest number    such that $(1,p_*)$  has this uniqueness property.  
  
\noindent
\textbf{Claim.}  We have that $p_*= 2^*_s-1$. \\
\noindent
 Suppose by contradiction that   $p_*<2^*_s-1$.    \\
We fix $\ov p\in (1, p_*)$ and $\b\in \left(0, \min\{(\ov p -1)s, s\}\right)$. We introduce the Banach space\footnote{by $\Ds u\in  {C^\b( {\ov{B}})} $,  we mean there exists $f\in C^\b({\ov{B}} )$ such that  $\Ds u=f$ in $\calD'{(B)}$. We note in this case that $\Ds u(x)=f(x)$ for all $x\in {B}$ because $u\in C^{2s+\b}_{loc}({B})\cap C^s({\ov{B}})$ by regularity theory.} 
\be 
 \label{eq:defcalC2sb}
 \calC^{2s+\b}_0:=\left\{u\in C^{\b}(\R^N), \quad\quad  u=0 \textrm{ in $\R\setminus {B}$,}\quad  \Ds u\in  {C^\b( {\ov{B}})}  \right\}, 
\ee
endowed with the norm $\|u\|_{C^\b(\R)}+\|\Ds u\|_{C^\b({\ov{B}})}$.    
Note that, by  Lemma \ref{lem:qual-sol},   all solutions to \eqref{eq:1.1} belongs to  $ \calC^{2s+\b}_0$. 
We finally define 
$$
F: \left( 1,  \infty\right)\times   \calC^{2s+\b}_0 \to C^\b({\ov{B}}), \qquad F(p,u)=\Ds u  - |u|^{p} .
$$
It is easy to see  that $F$ is of class $C^1$ on $(1, \infty)\times  \calC^{2s+\b}_0$.  
 We have that   $F(p_*, u_{p_*})=0$  and $\de_u F(p_*,u_{p_*})= \Ds   -p_*   u^{p_*-1}_{p_*}$, which has empty kernel by  Theorem~\ref{radial-nondegeneracy-combined}(ii).  It is easily seen that  $\Ds: \calC^{2s+\b}_0 \to C^\b({\ov{B}})$ is a Fredholm map of index zero.  In addition since $u_{p_*}\in C^s({\ov{B}})$ we have that $u_{p_*}^{p_*-1}\in C^\b({\ov{B}})$ by our choice of $\b$ and also  by the Arz\`ela-Ascoli theorem the map     $v\mapsto u^{p_*-1}_{p_*} v: \calC^{2s+\b}_0 \to C^\b({\ov{B}})$ is compact\footnote{because if  $\|\Ds v_n \|_{C^\b(\ov B)}$ is bounded then   $\|v_n\|_{ C^s({\ov{B}})}$ and thus $(v_n)$ has a convergent subsequence in $C^\b({\ov{B}})$ by the choice of $\b<s$.}.  As a consequence,   $\de_u F(p_*,u_{p_*}) : \calC^{2s+\b}_0 \to C^\b({\ov{B}})$ is an isomorphism.
It then follows from the implicit function theorem that there exists $\d>0$ such that for all $p\in (p_*-\d,p_*+\d)$, there exists a unique $u_p\in B_{ \calC^{2s+\b}_0 }(u_{p_*},  \d)$ satisfying  $F(p, u_p)= 0$.  Suppose that  there exists an  other   $\ti  u_{p_*}\in \calC^{2s+\b}_0\setminus \{0\}$  satisfying $F(p_*, \ti u_{p_*})=0$.  Similarly, by nondegeneracy from Theorem~\ref{radial-nondegeneracy-combined}(ii), decreasing  $\d$ if necessary,  we have  that for all $p\in (p_*- \d,p_*+ \d)$, there exists a unique $\ti u_p\in B_{ \calC^{2s+\b}_0 }(\ti u_{p_*},   \d)$ satisfying  $F(p, \ti u_p)= 0$.   Note that taking $\d$ smaller if necessary,  $u_p(0)>0$ and $\ti u_{p}(0)>0$   for all $p\in (p_*-\d,p_*,+\d)$,  thanks to  the continuity of the curves $p\mapsto u_p$ and $p\mapsto \ti u_p$ as maps $(p_*-\d,p_*+\d)\to \cC_0^{2s+\b} $, which follows from the implicit function theorem. The maximum principle then implies that $u_p>0$ in $B$ and $\ti u_p>0$ in $B$ and  they  satisfy $F(p,u_p)=F(p,\ti u_p)=0$ for all $p\in (p_*-\d,p_*+\d)$.
From the definition of $p_*$, we have that $u_{p}=\ti u_p$ for all $p\in (p_*-\d,p_*)$.    We can let $p \nearrow p_*$ and we obtain $u_{p_*}=\ti u_{p_*}$  This implies that $u_{p_*}$ is the  unique solution of $F(p_*,\cdot)=0$ in $\calC^{2s+\b}_0\setminus \{0\}$.\\

To obtain a contradiction on the assumption on $p_*$,  we prove that there exists $\e_*\in (0, \d)$ such that for all  $p\in (p_*,p_*+\e_*)$ the equation $F(p_*,\cdot)=0$  possesses a unique  solution in $ \calC^{2s+\b}_0\setminus \{0\}$.   
 If such an $\e_*$ does not exist,   then  we can find  a sequence  $(p_j)_{j\in \N} $ with  $ p_j\searrow p_*$ such that for all $j\in\N$, there exist   $u_{p_j}, \ti u_{p_j} \in \calC^{2s+\b}_0\setminus \{0\}$  satisfying $F(p_j,u_{p_j})=F(p_j,\ti u_{p_j})=0 $     with the property that     $u_{p_j}\not=\ti u_{p_j}$.    
  By Lemma \ref{lem:unif-bnd},   both  $ u_{p_j}$ and $\ti u_{p_j}$ converge to $u_{p_*}$ in $\cH^s{(B)}\cap C(\R)$ by the uniqueness of $u_{p_*}$.  We next note that $\th_j:=\frac{u_{p_j}-\ti u_{p_j}}{\|u_{p_j} -\ti u_{p_j} \|_{L ^2{(B)}}}$
satisfies $\Ds \th_j=g_j \th_j$ in ${B}$, with 
$$
g_j=p_j  \int_0^1( t u_{p_j} +(1-t) \ti u_{p_j})^{p_j-1}dt,
$$
and $\|\th_j\|_{L^2(B)}=1$.   By Lemma \ref{lem:unif-bnd}, $[\th_j]_s\leq \|g_j\|_{L^\infty{(B)}}$ is uniformly bounded for all    $j$. Hence,  up to a subsequence,  we see that $\th_{j}$ weakly converges in $\cH^s{(B)} $ and strongly in $L^2(B)$ to some $\ov\th$ in $\cH^s{(B)} $ satisfying  $\|\ov \th\|_{L^2(B)}=1$ and 
$$
\Ds\ov \th -p_*  u_{p_*}^{p_*-1} \ov \th=0 \quad \textrm{ in ${B}$.} 
$$
This is impossible by Theorem~\ref{radial-nondegeneracy-combined}(ii), and thus $\e_*$ exists, contradicting the definition of $p_*$. We thus get $p_*+1=2^*_s$ as claimed.
\end{proof}

\begin{proof}[Proof of  Theorem~\ref{new-uniqueness-intervall-GS}(completed)]
  The proof follows the proof of Theorem~\ref{new-uniqueness-intervall} almost line by line.  Here, we let $p_*>1$ be the largest number such that for every $p \in (1,p_*)$  there exists a unique positive and nondegenerate ground state solution of (\ref{eq:1.1-simplified-GS}). As in the proof of Theorem~\ref{new-uniqueness-intervall}, we deduce from \cite{DIS} that such a value $p_*$ exists, and we get
$p_*+1=2^*_s$ in the same way by using now Theorem~\ref{radial-nondegeneracy-combined}(i) in place of Theorem~\ref{radial-nondegeneracy-combined}(ii).
\end{proof}

\begin{proof}[Proof of Theorem \ref{new-uniqueness-AT} (completed)]
The asserted radial nondegeneracy property of solutions of (\ref{eq:1.1-simplified-intervall-p2B}) and (\ref{eq:1.1-simplified-intervall-p2R}) are given by Theorem \ref{radial-nondegeneracy-p2B}   and Theorem \ref{radial-nondegeneracy-p2R}.  For the proof of uniqueness of problem \eqref{eq:1.1} in the case  $B=(-1,1)$ and $p=2$,  we  argue as in the proof of Theorem~\ref{new-uniqueness-intervall} by considering the $C^1$ function 
$$
G:\R \times   \calC^{2s+\b}_0 \to C^\b({\ov{B}}), \qquad G(\l, u)=\Ds u +\l u - u^2 ,
$$
 where $\b\in (0,s)$ and   $ \calC^{2s+\b}_0$ is defined by \eqref{eq:defcalC2sb}.   We start by showing the existence of   a local  branch    of unique solutions for $\l>0$ small,  using Theorem~\ref{new-uniqueness-intervall}.\\
\noindent
\textbf{Claim.} There exists $\l_*>0$ such that for $\l\in (-\l_*,\l_*)$ there exists a unique $u_\l \in \calC^{2s+\b}_0  \setminus \{0\}$ satisfying $G(\l, u_\l)=0$. \\
\noindent
Suppose on the contrary that the statement in the claim dose not hold.  Then there exist  sequences $\l_n\to 0$ and  $u_n, v_n \in \calC^{2s+\b}_0  \setminus \{0\}$ with $u_n\not= v_n$ such that  $G(\l_n,  u_n)=G(\l_n,  v_n)=0$.   Then by  Lemma \ref{lem:unif-bnd} we have that the sequences $u_n$ and $v_n$  converge in $\cH^s(B)\cap C^s(\R^N)$,  respectively, up to subsequences,     to some functions $u, v\in \calC^{2s+\b}_0  \setminus \{0\}$ satisfying $G(0,u)=G(0,v)=0$.  Hence by Theorem~\ref{new-uniqueness-intervall},  we have $u=v$.  On the other hand  the function $w_n:=\frac{u_n-v_n}{\|u_n-v_n\|_{L^2(B)}}\in \cH^s(B) $ satisfies
$$
\Ds w_n+\l_n w_n=g_nw_n \qquad\textrm{ in $B$,}
$$
with $g_n:=   u_n+v_{n} $ which is bounded in $C^s(\R)$.
Hence, up to a subsequence,  $w_n$ weakly (resp. strongly) converges in $\cH^s(B) $ (resp. in $L^2(B)$)  to some  $w$ satisfying  $\Ds w=2 uw$ in $B$  and $\|w\|_{L^2(B)}=1$.  This is  in contradiction with Theorem~\ref{new-uniqueness-intervall}.  The claim is thus proved.\\
To continue the unique local branch, we recall that, in view of Theorem \ref{radial-nondegeneracy-p2B},  $\de_u G(\l, u)$ is an isomorphism for all $\l>0$ and $u\in \calC^{2s+\b}_0  \setminus \{0\}$.   We can now argue as in the proof of Theorem~\ref{new-uniqueness-intervall} to conclude that the largest $\l_*>0$ for which for all $\l\in (0,\l_*)$  the equation  $G(\l,\cdot)=0$ has a unique solution in $\calC^{2s+\b}_0  \setminus \{0\}$   is $+\infty$. This completes the uniqueness of the solution to \eqref{eq:1.1} in the case $p=2$.\\
\\
%
In the case of   \eqref{eq:1.1RN},  the proof of uniqueness   follows the argument of \cite[Section 8]{FLS}, replacing ground state with solution. Note that Theorem \ref{radial-nondegeneracy-p2R}  implies that the hypothesis of     \cite[Proposition 8.1]{FLS} are verified for all $s>\frac{1}{6}$, $N=1$ and $p=2$.   One  therefore repeat the proof of  \cite[Proposition 8.4]{FLS}, replacing ground state solution with  any even solution  to  \eqref{eq:1.1RN}, to conclude uniqueness of even solutions to \eqref{eq:1.1RN}  in $H^s(\R)$.
 \end{proof}

\section{The special case $s = \frac{1}{2}$ and $p = 1+ \frac{2}{N+1}$}
\label{special-case}
This section is devoted to the proof of   Theorem~\ref{new-gagliardo-nirenberg-uniqueness}. 
More precisely, we first determine the unique positive radial ground state solution of $(-\Delta )^{\frac{1}{2}} u+ u=u^{p}$ in $\mathbb{\R}^N$, $N \ge 1$ in the special case $p = 1+ \frac{2}{N+1}$. Then, we compute the best constants in the associated Sobolev and Gagliardo-Nirenberg inequalities. We start with the following lemma.
 
\begin{lemma}\label{lem:expli-sol}
 For $a>0$, let $u_a \in H^{1/2}(\R^N) \cap C^\infty(\R^N)$ be defined by 
\be\label{eq:expl-sol}
u_a(x)= c_N \frac{a}{\left( a^2+ |x|^2 \right)^\frac{N+1}{2}},  \qquad\textrm{  with  $c_N=\frac{\G((N+1)/2)}{\pi^{(N+1)/2}}$}.
\ee
Then we have 
\begin{equation}
  \label{eq:explicit-equation}
 (-\D)^{\frac{1}{2}} u_a+ \frac{1}{a}  u_a=a^{2-p}\a_N u^p_a, \qquad 
\end{equation}
 with $p:=1+\frac{2}{N+1}$ and
 \begin{equation}
   \label{eq:def-alpha-N}
\a_N:=(N+1) c_N^{-\frac{2}{N+1}}= \frac{(N+1) \pi}{\G((N+1)/2)^{\frac{2}{N+1}}}.
  \end{equation}
\end{lemma}
\begin{proof}
By \cite[Theorem 1.14]{Stein-Weiss}, we have 
\be\label{eq:Four-e}
\cF (e^{-2\pi a |\cdot|})(x) = c_N \frac{a}{\left(a^2+ |x|^2 \right)^\frac{N+1}{2}}\qquad \text{for $x \in \R^N$}
\ee
with $c_N$ given in \eqref{eq:expl-sol}, where $\cF$ denotes the standard Fourier transform \footnote{We are considering $\cF(f)(\xi):=\int_{\R^N}f(x)e^{-2\pi \i \xi \cdot x}dx$ and   $\cF^{-1}(f)(x):=\int_{\R^N}f(\xi)e^{2\pi \i x \cdot \xi}d\xi$ so that $\cF\circ \cF^{-1}=\rm{id}.$}. 
Therefore 
\begin{align*}
\cF(  (-\D)^{\frac{1}{2}}  u_a)(\xi)=2\pi|\xi|e^{-2\pi a |\xi|}=- \frac{d}{d a} e^{-2\pi a |\xi|}.
\end{align*}
Taking inverse Fourier transform, we obtain 
\begin{align*}
  (-\D)^{\frac{1}{2}} u_a(x)&= -\frac{d}{d a}\cF^{-1}(e^{-2\pi a |\xi|})(x)= -\frac{d}{d a} u_a(x),
\end{align*}
where, by \eqref{eq:expl-sol},
$$
\frac{d}{d a} u_a(x) = c_N \frac{1}{\left( a^2+ |x|^2 \right)^\frac{N+1}{2}} - {(N+1)} c_N \frac{a^2}{\left( a^2+ |x|^2 \right)^{\frac{N+1}{2}+1}}= \frac{u_a(x)}{a} - a^{2-p}\a_N u^p_a(x)
$$
with $p:=1+\frac{2}{N+1}$ and $\a_N$ given in (\ref{eq:def-alpha-N}). Combining these identities gives (\ref{eq:explicit-equation}).
\end{proof}

\begin{lemma}\label{lem:xdotnu}
  Let $u= \alpha_N^{\frac{1}{p-1}} u_1$, i.e.,
  \begin{equation}
    \label{eq:definitionu-explicit}
  u(x)= \alpha_N^{\frac{N+1}{2}}u_1(x) = \frac{(N+1)^{\frac{N+1}{2}}}{c_N}u_1(x)= \frac{(N+1)^{\frac{N+1}{2}}}{\left(1+ |x|^2 \right)^\frac{N+1}{2}},
  \end{equation}
  where $u_1$ is given by \eqref{eq:expl-sol} with $a=1$. Then we have 
\begin{equation}
  \label{eq:explicit-equation-special-case}
 (-\D)^{\frac{1}{2}} u+ u= u^p 
\end{equation}
and
\begin{equation}
  \label{radial-gradient-equation}
x \cdot \n u=- (N+1) u+   u^{p}.
\end{equation}
\end{lemma}

\begin{proof}
The equation (\ref{eq:explicit-equation-special-case}) immediately follows from~(\ref{eq:explicit-equation}). Moreover, we have 
\begin{align*}
  x \cdot \n u &=-(N+1)^{1+\frac{N+1}{2}} \frac{ |x|^2}{\left( 1+ |x|^2 \right)^{\frac{N+3}{2}}}\\
&=-(N+1)^{1+\frac{N+1}{2}} \Bigl(\frac{1}{\left( 1+ |x|^2 \right)^{\frac{N+1}{2}}} -\frac{1}{\left( 1+ |x|^2 \right)^{1+\frac{N+1}{2}}}\Bigr)=- (N+1) u+  u^{p}.
\end{align*}
\end{proof}

From now on, we fix $u$ as given in (\ref{eq:definitionu-explicit}), and we consider the eigenvalues   $0< \L_1^r<\L_2^r\leq \dots$ of the weighted radial eigenvalue problem
\be\label{eq:weighted-eigen-R^N-section}
 (-\Delta)^{\frac{1}{2}} w+w= \L u^{p-1} w \qquad w\in H^{1/2}_{r}(\R^N).
\ee
Recall that the first radial eigenvalue $\L_2^r=1$ with eigenspace spanned by $u$.
\begin{proposition}
\label{proposition-radial-nondeg-explicit}  
The second radial eigenvalue $\L_2^r$ in the weighted eigenvalue problem  \eqref{eq:weighted-eigen-R^N-section}   satisfies $\L_2^r>p$.
\end{proposition}

\begin{proof}
We consider the linearized operator $L = (-\Delta)^{\frac{1}{2}}  + 1 - p u^{p-1}$ and note that
$$
L v = - u \qquad \text{with}\quad  v = \frac{1}{p-1}u + x \cdot \nabla u.
$$
This follows from (\ref{eq:explicit-equation-special-case}) and the well known identity
$$
(-\Delta)^s [x \cdot \nabla u] = 2s (-\Delta)^s u + x \cdot \nabla [(-\Delta)^s u] \quad \text{for $s \in (0,1)$.} 
$$
Let $w \in H^{1/2}_{r}(\R^N)$ be an eigenfunction  corresponding to $\L_2^r$, which implies that
\begin{equation}
 \label{L-u-w-equation}
L w = (\Lambda_2^r-p)u^{p-1}w.
\end{equation}
Then, since $u$ is an eigenfunction corresponding to $\L_1^r=1<\L_2^r$, we have the orthogonality property 
\be \label{eq:perp}
\int_{\R^N} u^{p}w\,dx = 0.
\ee
Moreover, by \cite[Theorem 2]{FLS}, the function $w$ changes sign precisely once in the radial variable, say at $r_0 \in (0,\infty)$. From this, \eqref{eq:perp} and the fact that $u$ is strictly decreasing in the radial variable, we deduce that  
 \be\label{eq:L2-neg}
 0= \int_{\R^N} u^{p}w\,dx> u^{p-1}(r_0)  \int_{\R^N} w u dx \quad \text{and therefore}\quad \int_{\R^N} w u dx<0.
 \ee
Using (\ref{L-u-w-equation}) and \eqref{eq:L2-neg}, we find that 
\begin{align}\label{eq:integ00}
0< -  \int_{\R^N} u w \,dx &= \int_{\R^N} (Lv)w dx =\int_{\R^N}v (Lw) dx = (\Lambda_2^r-p)\int_{\R^N} u^{p-1}v w\,dx  \nonumber\\
                   &= (\Lambda_2^r-p)\frac{1}{p-1}\int_{\R^N} u^{p} w\,dx +(\Lambda_2^r -p)\int_{\R^N} u^{p-1} (x \cdot \nabla u)w\,dx \nonumber\\
                   & = (\Lambda_2^r -p)\int_{\R^N} u^{p-1} (x \cdot \nabla u)w\,dx .
\end{align}
Moreover, from Lemma \ref{lem:xdotnu}  and \eqref{eq:perp}, we infer that 
\begin{equation}
  \label{eq:integ000}
\int_{\R^N} u^{p-1} (x \cdot \nabla u)w\,dx = \int_{\R^N} w u^{2p-1}dx >  u^{p-1}(r_0)\int_{\R^N} w u^{p}dx=0.
 \end{equation}
Here we used again that $u$ is strictly decreasing in the radial variable. Combining (\ref{eq:integ00}) and (\ref{eq:integ000}), we deduce that $\Lambda_2^r -p>0$, as claimed.
 \end{proof}

The following proposition completes the first part of the proof of Theorem~\ref{new-gagliardo-nirenberg-uniqueness}. 

 \begin{proposition}
\label{sec:special-case-s}
The function $u$ given in (\ref{eq:definitionu-explicit}) is the unique radial ground state solution of (\ref{eq:explicit-equation-special-case}), and hence it is the up to a constant unique minimizer of the subcritical fractional Sobolev quotient 
 $$
u \mapsto Q_{1/2}(u):= \frac{\|u\|_{1/2}^2}{\|u\|_{L^{p+1}(\R^N)}^2}.
 $$
 Moreover, the associated best Sobolev constant is given by
 \begin{equation}
   \label{eq:Sobolev-constant}
 S:= \inf_{u \in H^{1/2}(\R^N) \setminus \{0\}}Q_{1/2} (u)= (N+1)\Bigl( \frac{(N+2) \pi^{N/2}\Gamma(\frac{N}{2}) }{4  (N+1) \Gamma(N)}  \Bigr)^{\frac{1}{N+2}}.
 \end{equation}
\end{proposition}

\begin{proof}
  Combining Theorem~\ref{th:nondeg} and Proposition~\ref{proposition-radial-nondeg-explicit}, we deduce that $u$ is a ground state solution of (\ref{eq:explicit-equation-special-case}), so it is unique by \cite[Theorem 4]{FLS}.
Moreover, it is a well known fact (see e.g. \cite{FLS}) that the infimum in (\ref{eq:Sobolev-constant}) is attained, and that every minimizer is, up to a constant, a ground state solution of (\ref{eq:explicit-equation-special-case}). Consequently, $u$ is the unique minimizer up to a constant.
Moreover, by (\ref{eq:explicit-equation-special-case}) we have$\|u\|_{1/2}^2 = \|u\|_{L^{p+1}(\R^N)}^{p+1}$ and therefore
 $$
 S:= \frac{\|u\|_{1/2}^2}{\|u\|_{L^{p+1}(\R^N)}^2}= \|u\|_{L^{p+1}(\R^N)}^{p-1} =\Bigl(\int_{R^N} u^{p+1}dx\Bigr)^{\frac{p-1}{p+1}} 
 $$
with 
$$
\int_{R^N} u^{p+1}\,dx =  \int_{\R^ N} \frac{(N+1)^{\frac{(p+1)(N+1)}{2}}}{\left(1+ |x|^2 \right)^{\frac{(p+1)(N+1)}{2}}}\,dx = (N+1)^{N+2} \int_{\R^ N} \frac{1}{\left(1+ |x|^2 \right)^{N+2}}\,dx.
$$
Furthermore, with polar coordinates and the change of variables $\rho = \frac{1}{1+r^2}$, we compute that 
\begin{align*}
  \int_{\R^ N} \frac{1}{\left(1+ |x|^2 \right)^{N+2}}\,dx &= |S^{N-1}| \int_0^\infty r^{N-1}(1+r^2)^{-N-2}\,dr\\
                                                          &= \frac{|S^{N-1}|}{2}  \int_0^1 r^{\frac{N+2}{2}} (1-r)^{\frac{N-2}{2}}\,dr= \frac{|S^{N-1}|\Gamma(\frac{N}{2}+2)\Gamma(\frac{N}{2})}{2 \Gamma(N+2)}\\
  &=\frac{|S^{N-1}| \frac{N+2}{2} \frac{N}{2} \Gamma^2(\frac{N}{2})}{2 \Gamma(N+2)} =\frac{(N+2) \pi^{N/2}\Gamma(\frac{N}{2}) }{4
    (N+1) \Gamma(N)}.
\end{align*}
Since $\frac{p-1}{p+1} = \frac{1}{N+2}$, we thus conclude that
$$
S= (N+1)\Bigl( \frac{(N+2) \pi^{N/2}\Gamma(\frac{N}{2}) }{4
  (N+1) \Gamma(N)}  \Bigr)^{\frac{1}{N+2}},
$$
as claimed.
\end{proof}

Next we wish to compute the best constant in the associated fractional Gagliardo-Nirenberg inequality. We note the following general lemma, which follows by similar arguments as in \cite[Proof of Theorem 1]{del-pino-dolbeault}. For the reader's convenience, we include a short proof.

 \begin{lemma}
\label{lem-Sobolev-Gagliardo-Nirenberg}   
Let $s \in (0,1)$, $q \in (2, \Bigl[\frac{2N}{N-2s}\Bigr]_+)$ and 
\begin{equation}
  \label{eq:minimization-sobolev}
 S_{s,q}:= \inf_{u \in H^s(\R^N) \setminus \{0\}}\frac{\|u\|_s^2}{\|u\|_{L^{q}(\R^N)}^2}>0.
\end{equation}
Moreover, let  $\theta = \frac{N(q-2)}{2qs}$. Then 
\begin{equation}
  {GN}_{s,q}:= \inf_{u \in H^s(\R^N) \setminus \{0\}}\frac{[u]_s^{\theta} \|u\|_{L^2(\R^N)}^{1-\theta}}{\|u\|_{L^{q}(\R^N)}}= \Bigl(\theta^{\theta}(1-\theta)^{(1-\theta)}S_{s,q}\Bigr)^{1/2} .  \label{eq:gagliardo-nirenberg}
\end{equation}
\end{lemma}

\begin{proof}
We consider the scaling
 $$
H^s(\R^N) \to H^s(\R^N), \qquad  u \mapsto u_\lambda = \lambda^{N/q}u(\lambda \:\cdot\:)
$$
which preserves the $L^q$-norm, i.e.
\begin{equation}
  \label{eq:L-q-invariance}
\|u_\lambda\|_{L^q(\R^N)} = \|u\|_{L^q(\R^N)} \qquad \text{for all $u \in H^s(\R^N)$.}
\end{equation}
Since, by a change of variable, $[u_\lambda]_s^2 = \lambda^{2N/q-N+2s}[u]_s^2$ and $
 \|u_\lambda \|_{L^2(\R^N)} = \lambda^{N/q-N/2} \|u\|_{L^2(\R^N)}$ for $u \in H^s(\R^N)$, it also follows from our special choice of $\theta$ that 
 $$
 [u_\lambda]_s^{\theta}\|u_\lambda\|_{L^2(\R^N)}^{1-\theta}  =  [u]_s^{\theta}\|u\|_{L^2(\R^N)}^{1-\theta} \qquad \text{for $u \in H^s(\R^N)$, $\lambda>0$.}
$$
Next, let $u \in H^s(\R^N)$ with $\|u\|_{L^q(\R^N)}=1$ be a minimizer for (\ref{eq:minimization-sobolev}) \footnote{Standard arguments based on Schwarz symmetrization and compactness imply that the infimum in (\ref{eq:minimization-sobolev})  is attained by a radial function in $u \in H^s(\R^N)$.}. The minimization property and (\ref{eq:L-q-invariance}) then imply that
 \begin{equation*}
   0 = \frac{d}{d\lambda}\Bigl|_{\lambda = 1}\|u_\lambda\|_s^2 = (2N/q-N+2s)[u]_s^2 + (2N/q-N)\|u\|_{L^2(\R^N)}^2,
 \end{equation*}
 hence
 $$
 (1-\theta) [u]_s^2 = \theta \|u\|_{L^2(\R^N)}^2.
 $$
Consequently,
 $$
 S_{s,q}=\|u\|_s^2 = [u]_s^2 + \|u\|_{L^2(\R^N)}^2= \frac{1}{\theta}[u]_s^2 = \frac{1}{1-\theta}\|u\|_{L^2(R^N)}^2
$$
and therefore 
$$
 [u]_{s}^{\theta}\|u\|_{L^2(\R^N)}^{(1-\theta)}=\Bigl(\theta^{\theta}(1-\theta)^{(1-\theta)}S_{s,q}\Bigr)^{1/2}.
 $$
 On the other hand, for a given function $u \in H^{s}(\R^N)$ with $\|u\|_{L^q(\R^N)}=1$, we may choose $\lambda>0$ with
 $$
 (1-\theta) [u_\lambda]_s^2 = \theta \|u_\lambda \|_{L^2(\R^N)}^2,
 $$
 which implies that 
 $$
 S_{s,q} \le \|u_\lambda\|_s^2 = \frac{1}{\theta}[u_\lambda]_s^2 =
\frac{1}{1-\theta} \|u_\lambda\|_{L^2(\R^N)}^2
 $$
and  hence
 $$
 [u]_{s}^{\theta}\|u\|_{L^2(\R^N)}^{(1-\theta)} =  [u_\lambda]_{s}^{\theta}\|u_\lambda \|_{L^2(\R^N)}^{(1-\theta)} \ge \Bigl(\theta^{\theta}(1-\theta)^{(1-\theta)}S_{s,q}\Bigr)^{1/2}.
 $$
 This shows (\ref{eq:gagliardo-nirenberg}), as claimed.
\end{proof}$ $

To complete the proof of Theorem~\ref{new-gagliardo-nirenberg-uniqueness}, it remains to apply Lemma~\ref{lem-Sobolev-Gagliardo-Nirenberg} in the special case $s = \frac{1}{2}$ and $q= p+1 = \frac{2(N+2)}{N+1}$, which gives $\theta  = \frac{N}{N+2}$ and $(1-\theta) = \frac{2}{N+2}$. Since we already now that
 $$
S = S_{s,q} = ({N+1})\Bigl(\frac{\pi^{\frac{N}{2}}(N+2)\Gamma(\frac{N}{2})}{4 (N+1)\Gamma(N)}\Bigr)^{\frac{1}{N+2}}
 $$
 in this case, it follows that
 \begin{align*}
   GN_{s,q} &= \Bigl[\Bigl(\frac{N}{N+2}\Bigr)^{\frac{N}{N+2}}\Bigl(\frac{2}{N+2}\Bigr)^{\frac{2}{N+2}}({N+1})\Bigl(\frac{\pi^{\frac{N}{2}}(N+2)\Gamma(\frac{N}{2})}{4 (N+1)\Gamma(N)}\Bigr)^{\frac{1}{N+2}}\Bigr]^{1/2}\\
&= \Bigl(\frac{N^N (N+1)^{N+1} \pi^{\frac{N}{2}}\Gamma(\frac{N}{2})}{(N+2)^{N+1}\:\Gamma(N)}\Bigr)^{1/2(N+2)} .  
 \end{align*}
 Since $C_{GN} = GN_{s,q}^{-1}$, we have thus completed the proof of Theorem~\ref{new-gagliardo-nirenberg-uniqueness}.
 
\section{Appendix}

\subsection{A geometric lemma}

In the following, we consider, as before, for $\nu \in S^{N-1}$, the reflection
$$
\sigma_\nu: \R^N \to \R^N,\qquad \sigma_\nu(x)= x-2 (x \cdot \nu)\nu
$$
with respect to the hyperplane $T_\nu:= \{x \in \R^N \::\: x \cdot \nu = 0\}$.

\begin{lemma}
\label{appendix-lemma-sphere}
  Let $w,h \in C^2(S^{N-1})$ be functions with the property that 
  \begin{equation}
    \label{eq:basic-assumption}
  \frac{w(z) - w(\sigma_\nu(z))}{2} = h(\nu)\, z \cdot \nu \qquad \text{for every $\nu,z \in S^{N-1}$.}
  \end{equation}
  Let, moreover, $\nu_{max} \in S^{N-1}$ be a point with $h_{max}:= h(\nu_{max}) = \max \limits_{S^{N-1}} h$. Then we have
  \begin{equation}
    \label{eq:appendix-lemma-conclusion}
    w(z) = [w(\nu_{max})- h_{max}] + h_{max}\: z \cdot \nu_{max} \qquad \text{for $z \in S^{N-1}$.} 
  \end{equation}
In particular, $w$ is a sum of a constant and an odd function with respect to the reflection $\sigma_{\nu_{max}}$.
\end{lemma}

\begin{proof}
  Since the problem is rotationally invariant, we may assume, without loss of generality, that $h$ takes its maximum at $\nu_{max} = e_N:= (0,\dots,0,1)$.\\
  We first consider the case $N=2$.  We then write $x_\theta = (\cos \theta,\sin \theta) \in S^1$ for $\theta \in \R$ and regard all functions as $2\pi$-periodic functions of the angle $\theta$. For a fixed angle $\theta$, we consider $\nu_\theta:= (-\sin \theta, \cos \theta)$. Then the reflection at the hyperplane $T = T_{\nu_\theta}:= \{x \in S^1\::\: x \cdot \nu_\theta = 0\}$ corresponds in the angle coordinate to the map
  $$
  \vartheta \mapsto 2 \theta - \vartheta
  $$
  Setting
  $$
  \tilde w(\theta) = w(x_\theta)= w(\cos \theta,\sin \theta) \qquad \text{and}\qquad
  \tilde h(\theta) = h(\nu_\theta) = h(-\sin \theta, \cos \theta),
  $$
  we may then reformulate assumption (\ref{eq:basic-assumption}) as
  \begin{align*}
    &\frac{w(\theta+s)-w(\theta-s)}{2} = \tilde h(\theta)\:x_\theta \cdot \nu_\theta =  \tilde h(\theta)\: (\cos (\theta+s),\sin (\theta+s)) \cdot (-\sin \theta, \cos \theta)\\
    &= \tilde h(\theta)
\Bigl(\sin(\theta+s)\cos \theta -\cos(\theta+s)\sin \theta\Bigr)
   = \tilde h(\theta)\sin s \quad \text{for $\theta, s \in \R.$}   
  \end{align*}
  Differentiating in $\theta$ at $s = 0$ gives
  \begin{equation}
    \label{eq:theta-deriv-formula}
  \partial_\theta w(\theta) = \tilde h(\theta) \qquad \text{for all $\theta \in \R$.}
  \end{equation}
  Consequently, we may reformulate the assumption again as
  $$
  \frac{1}{2}\int_{\theta-s}^{\theta+s}\tilde h(\tau)d\tau = \tilde h(\theta) \sin s \quad \text{for $\theta, s \in \R.$}   
  $$
  Differentiating this identity three times in $s$ gives
  $$
  \frac{\tilde h''(\theta+s)+\tilde h''(\theta-s)}{2} = -\tilde h(\theta)\cos s  \quad \text{for $\theta, s \in \R.$} 
$$
Evaluating at $s = 0$ gives $\tilde h''(\theta) = -\tilde h(\theta)$, from which we deduce that 
\begin{equation}
    \label{eq:theta-deriv-formula-1}
\tilde h(\theta) = h_{max} \cos \theta \qquad \text{for $\theta \in \R$.}
\end{equation}
Here we used the fact that $\tilde h$ takes its maximum at zero, since $h$ takes its maximum at $\nu_{max} = (0,1)$ by assumption. Combining (\ref{eq:theta-deriv-formula}) and (\ref{eq:theta-deriv-formula-1}) gives 
\begin{align*}
w(x_\theta) &= \tilde w(\theta) = \tilde w(\frac{\pi}{2}) + h_{max} \int_{\frac{\pi}{2}}^\theta \cos \vartheta d\vartheta 
= \tilde w(\frac{\pi}{2}) + h_{max} (\sin \theta-1)\\
&= w (\nu_{max}) + h_{max} (x_\theta \cdot \nu_{max}-1) \qquad \text{for $\theta \in \R$,}
\end{align*}
which gives (\ref{eq:appendix-lemma-conclusion}) in the case $N=2$.\\
In the general case $N \ge 2$, we may repeat the above argument in the $2$-dimensional subspace $\spann \{e_*, e_N\}$, where $e_* \in S^{N-2} \times \{0\}$ is chosen arbitrarily. This then yields (\ref{eq:appendix-lemma-conclusion}) for arbitrary $z \in S^{N-1}$.
\end{proof}

\begin{corollary}
\label{appendix-corollary-sphere}
  Let $w \in C^2(\R^N)$, $h \in C^2(S^{N-1})$ and $U:(0,\infty) \to [0,\infty)$ be functions with the property that 
  \begin{equation*}
  \frac{w(x) - w(\sigma_\nu(x))}{2} =  h(\nu)U(|x|) \, \frac{x}{|x|} \cdot \nu \qquad \text{for every $\nu \in S^{N-1}, x \in \R^N \setminus \{0\}$.}
  \end{equation*}
  Let, moreover, $\nu_{max} \in S^{N-1}$ be a point with $h_{max}:= h(\nu_{max}) = \max \limits_{S^{N-1}} h$. Then we have
  \begin{equation*}
    w(x) = [w(|x| \nu_{max})- h_{max}\,U(|x|)] +  h_{max}\, U(|x|) \: \frac{x}{|x|} \cdot \nu_{max} \qquad \text{for $x \in \R^{N} \setminus \{0\}$.} 
  \end{equation*}
In particular, $w$ is a sum of a radial function and an odd function with respect to the reflection $\sigma_{\nu_{max}}$.
\end{corollary}

\begin{proof}
  It suffices to apply Lemma~\ref{appendix-lemma-sphere}, for fixed $r>0$, to the functions
  $$
  S^{N-1} \to \R,\qquad z \mapsto w(r z)
  $$
  in place of $w$ and
  $$
  S^{N-1} \to \R,\qquad \nu \mapsto U(r)h(\nu)
  $$
  in place of $h$.
\end{proof}

\subsection{A Hopf type lemma}

In this section, we provide a variant of the fractional Hopf type lemma in \cite[Prop. 2.2]{SV} under somewhat weaker assumptions.
\begin{lemma}
  \label{hopf-type-lemma}
  Let $H_+:= \{x \in \R^N\::\: x_1>0\}$, $T:= \{x \in \R^N\::\: x_1 = 0\}$, and let $\Omega \subset \R^N$ be a bounded set of class $C^2$ which is symmetric with respect to reflection of the $x_1$-coordinate, and let $\Omega_+:= \Omega \cap H_+$.   Moreover, let $c \in L^\infty(\Omega)$,  $\alpha>0$ and  $v \in L^1(\R^N;(1+|x|)^{-N-2s}) \cap C^{2s+\alpha}(\Omega)$ satisfy
  \begin{equation}
    \label{eq:hopf-eq}
    \left\{
      \begin{aligned}
      (-\Delta)^s v + c(x) v \ge 0 \qquad \text{in $\Omega_+$,}\\
      v \ge 0 \qquad \text{in $H_+,$}\\
      v \not \equiv 0 \qquad \text{in $\Omega_+$,}\\
      v(-x_1,x')=-v(x) \qquad \text{for $x = (x_1,x') \in \R^N$.}
      \end{aligned}
    \right.
  \end{equation}
  Then we have
  $$
  v >0 \quad \text{in $\Omega_+$}\qquad \text{and}\qquad \liminf_{t \to 0^+} \frac{v(t,x')}{t} >0 \qquad \text{for every $(0,x') \in T \cap \Omega$.}
  $$
\end{lemma}

\begin{proof}
  In the case where, in addition, $v \in H^s(\R^N)$, the conclusion is contained in \cite[Cor. 3.3]{FJ-2015} and \cite[Prop. 2.2]{SV}. In fact, \cite[Prop. 2.2]{SV} also assumes the equality $(-\Delta)^s v + c(x)=0$ in $\Omega^+$, but the proof given there only requires that $(-\Delta)^s v + c(x) \ge 0$ in $\Omega^+$. To get rid of the additional assumption $v \in H^s(\R^N)$, we use a cut-off argument. So we let $v \in L^1(\R^N;(1+|x|)^{-N-2s}) \cap C^{2s+\alpha}(\Omega)$ satisfy \eqref{eq:hopf-eq}, and we let $\tilde x \in \Omega$ with $\tilde x_1 \ge 0$. We choose a radial cut-off function $\psi \in C^\infty_c(\Omega)$ with $0 \le \psi \le 1$ in such a way that there exists a (sufficiently large) neighborhood $U \subset \Omega$ of $\tilde x$ with $\psi \equiv 1$ in $U$ and $v \not \equiv 0$ in $U_+: = U \cap H_+$. We may assume also that $U$ is symmetric with respect to the $x_1$-coordinate. For $x \in U_+$ we then have
  \begin{align*}
    (-\Delta)^s(\psi v)(x) =   (-\Delta)^s v(x) +   (-\Delta)^s((\psi-1) v)(x) &\ge - c(x)v(x) + c_{N,s} f(x)\\
    &= - c(x)(\psi v)(x) + c_{N,s} f(x),
  \end{align*}
    where
    $$
    f(x)=  \int_{\R^N  \setminus U}\frac{(1-\psi)(y)v(y)}{|x-y|^{N+2s}}\,dy = \int_{(\R^N  \setminus U)\cap H_+}(1-\psi)(y)v(y)\Bigl(\frac{1}{|x-y|^{N+2s}}-\frac{1}{|x-\bar y|^{N+2s}}\Bigr) \,dy \ge 0
    $$
    and we write $\bar y = (-y_1,y')$ for $y=(y_1,y') \in \R^N$. Here we have used the oddness of the function $y \mapsto (1-\psi)(y)v(y)$. Consequently, we have 
  \begin{equation*}
    \begin{aligned}
      (-\Delta)^s (\psi v) + c(x) (\psi v) \ge 0 \qquad \text{in $U_+$,}\\
      \psi v \ge 0 \qquad \text{in $H_+,$}\\
      \psi v \not \equiv 0 \qquad \text{in $U_+$,}\\
      (\psi v)(-x_1,x')=-(\psi v) (x) \qquad \text{for $x = (x_1,x') \in \R^N$.}
      \end{aligned}
  \end{equation*}
  Since moreover $\psi v \in H^s(\R^N)$, \cite[Cor. 3.3]{FJ-2015} gives $v(\tilde x) = \psi v(\tilde x)>0$ if $\tilde x \in \Omega_+$,  and \cite[Prop. 2.2]{SV} gives 
  $$
  \liminf_{t \to 0^+} \frac{v(t,x')}{t} =   \liminf_{t \to 0^+} \frac{(\psi v)(t,x')}{t} >0 \quad \text{if $\tilde x = (0,x') \in \Omega \cap T$.}
  $$
The claim thus follows.
\end{proof}

\subsection{A topological lemma}
\label{sec:topological-lemma}

As before, we let $\R^2_+:= \{(x,t) \in \R^2\::\: t>0\}$. The following topological lemma was used in the proof of Theorem~\ref{radial-nondegeneracy-combined}.

\begin{lemma}
  \label{sec:topological-lemma-1}
  Let $x_1< x_2 <x_3 < x_4$ be real numbers. Suppose that $\gamma, \eta: [0,1] \to \overline{\R^2_+}$ are continuous curves such that  $\gamma(0)=(x_1,0)$, $\gamma(1)=(x_3,0)$, $\eta(0)=(x_2,0)$, $\eta(1)=(x_4,0)$. Then the curves $\gamma$ and $\eta$ intersect in $\overline{\R^2_+}$, i.e. there exists $t,\tilde t \in (0,1)$ with $\gamma(t)= \eta(\tilde t)$.
\end{lemma}

This lemma is a variant of \cite[Lemma D.1]{FL}. In contrast to \cite[Lemma D.1]{FL}, we do not assume that the curves $\gamma$ and $\eta$ are simple (i.e., injective), and we do not require that $\gamma(t), \eta(t) \in \R^2_+$ for $t \in (0,1)$. Clearly, under the more general assumptions of Lemma~\ref{sec:topological-lemma-1}, the intersection point may also lie on $\partial \R^2_+$. A full proof of Lemma~\ref{sec:topological-lemma-1}, based on the intersection number of two curves, can already be found in \cite[p. 554--556]{GPW}, and it is very likely that there are earlier references for this lemma which we are not aware of.

\end{document}